\documentclass[]{article}      
\pdfoutput=1

\usepackage{graphicx}

 \usepackage{latexsym,amsmath,amssymb,amsthm,textcomp}

 \newtheorem{theorem}{Theorem}[section]
\newtheorem{lemma}{Lemma}[section]
 \newtheorem{proposition}{Proposition}[section]
 \newtheorem{corollary}{Corollary}[section]
  \newtheorem{rem}{Remark}[section]

\begin{document}

\title{{\LARGE   Double scale analysis of periodic  solutions of some
 non linear  vibrating  systems
     } } 
%\titlerunning{ Double scale analysis of periodic  solutions}
 \author{
Nadia Ben Brahim \\
  University of Tunis, El Manar,  Ecole nationale
  d'ing\'enieurs de Tunis (ENIT), \\
Laboratory of  civil engineering (LGC), 
BP 37, 1002 Tunis Belv\'ed\`ere,
Tunisia\\ 
 and
Bernard Rousselet \\
 University of  Nice Sophia-Antipolis, Laboratory J.A. Dieudonn\'e \\ 
U.M.R.  C.N.R.S. 6621, Parc Valrose,  F 06108 Nice,
 C{e}dex 2, France\\
email: {\sf br\char64math.unice.fr} 
 }
\date{june 13, 2013}

\maketitle
\bibliographystyle{alpha}

\begin{abstract}
 We consider {\it small solutions} of a vibrating system with
smooth non-linearities for which we provide an approximate solution by
using a double scale analysis; a rigorous proof of convergence of a
double scale expansion is included; for the forced response, a stability
result is needed in order to prove convergence in a neighbourhood of a
primary resonance.
%\keywords{double scale analysis; periodic solutions; nonlinear
%  vibrations, resonance}
%\subclass{34e13, 34c25, 74h10, 74h45}
\end{abstract}

\section{Introduction}
\label{sec:intro}
%{ \bf to be continued}

In this work we look for an asymptotic expansion of {\it small periodic solutions} of free vibrations of a discrete
structure without damping and  with local non linearity; then the
same system with light damping and  a {\it periodic forcing}
with frequency close to a frequency of the free system is
analyzed (primary resonance). For a small solution, we recover a behavior with some
similarity with the linear case; in particular the amplitude of the
forced response reaches a local maximum at the frequency of the free
response. On the other hand the frequency of the free response is
amplitude dependent and the superposition principle does not apply. 
The work of Lyapunov {\cite{lyapunov49}} is often cited as a basis for the
existence of periodic solutions which tends towards linear normal modes as
amplitudes tend to zero; the proof of this paper uses the \textit{hypothesis of
analycity} of the non linearity involved in the differential system. In
\cite{rousselet:hal-perio-lip}, we addressed the case of a non
linearity which is only lipschitzian and we prove existence of
periodic solutions with a constructive proof; in this case the result of Lyapunov
obviously may not be applied. 
Non-linearity of oscillations is a classical theme in theoretical
physics, for example at master level, see \cite{MR0102191} in Russian
or its English or French translation in \cite{MR0120782,MR0205515}.

Asymptotic expansions have been used for
a long time; such methods are introduced in the famous memoir of Poincar\'e
\cite{poincare92-99}; a general book on asymptotic methods is
\cite{Bogolyu-Mitropo-ru} with french and English translations \cite{Bogolyu-Mitropo-fr,Bogolyu-Mitropo-eng}; introductory material is in \cite{nayfeh81},   \cite{miller2006}; 
 a detailed account of the
averaging method with precise proofs of convergence may be found in
\cite{sanders-verhulst}; an analysis of several methods including
multiple scale expansion may be found in \cite{murdock91};
the case of vibrations with
unilateral springs have been presented in \cite{sj-brgdr08,junca-br10,vestroni08},
\cite{hr-brgdr08,hazim-tamtam09,hazim-ecssmt,hazimSmai,hazim-these}; in
\cite{jiang-pierre-shaw04} a numerical approach for large solutions of
piecewise linear systems is proposed.  The case of  rigid  contact which is also important from the point of view of theory and applications has been addressed in several papers,
 for example \cite{janin-lamarque01}, and a synthesis in \cite{Bastien-Bernard-lamarque}
.
A review paper for so called
``non linear normal modes'' may be found in \cite{nnm-kpgv}; it includes numerous
papers published by the mechanical community; several application
fields have been addressed by the mechanical community; for example in \cite{mikhlin10} ``nonlinear
vibro-absorption problem, the cylindrical shell nonlinear dynamics and
the vehicle suspension nonlinear dynamics are analyzed''.

In the mechanical engineering community  the
validity of the expansions is assumed to hold; however, this is not
straightforward as this kind of expansion is not a standard series
expansion and the expansion is usually not valid for all time; for example, this point has
been raised in \cite{rubenfeld78}. If the averaging method was
carefully analyzed as indicated above, it seems not to be the case for
the multiple scale method, the expansion of which is often compared to
the one obtained by the averaging method.

Here in a first stage we consider {\it small solutions} of a system with
smooth non-linearities for which we provide an approximate solution by
using a double scale analysis; a rigorous proof of convergence of the
method of
double scale is included; for the forced response, a stability
result is needed in order to prove convergence. As an introduction, the next section
addresses the  one degree of freedom case while the following one
considers many degrees of freedom; for free vibrations we find
solutions close to a linear normal mode (so called non linear normal
modes) and for forced vibrations, we describe the response for forcing
frequency close to a free vibration frequency.
Preliminary versions of these results may be found in
\cite{benbrahim-tamtam09} and have been presented in conferences
  \cite{benbrahimGdrafpac,benbrahimSmai}; related results have been
  presented in \cite{gasmiSmai}. Triple scale expansions is to be submitted
 \cite{nbb-br-mult}.
In a forthcoming paper, the non-smooth case will be considered as well
as a
numerical algorithm based on the fixed point method used in
\cite{rousselet:hal-perio-lip}.

\section{  One degree of freedom,  strong cubic non linearity}
In this section, we consider the case of a mass attached to a spring;
in the case of a stress-strain law of the form ${n}{}=k u+ m c u^2 +m
u^3$, we find no shift of frequency at first order, so here we
concentrate on a stress-strain law with a
stronger cubic non linearity:
$${n}{}=k u+ m c u^2 +m \frac{d}{\epsilon} u^3$$
where $\epsilon$ is a small parameter which is also involved in the size
of the solution as in previous paragraph; the choice of this scaling
provides frequencies which are amplitude dependent.
\subsection{ Free vibration,  double scale expansion up to first order}
Using second Newton law, free vibrations  of a mass attached to such a
spring are governed by:
\begin{equation}
  \label{eq:grosse-cub-lib}
  \ddot u + \omega^2 u +cu^2 +\frac{d u^3}{\epsilon}=0.
\end{equation}
We look for a {\it small} solution with a  double scale for time; we set  
\begin{equation}
  T_0=\omega t, \quad T_1=\epsilon t,  
\end{equation}
so with $D_0u=\frac{\partial u}{\partial T_0}, \quad D_1u=\frac{\partial u}{\partial T_1}$,
we obtain
\begin{equation}
  \frac{du}{dt}=\omega D_0u +\epsilon D_1 u,  \;  \frac{d^2u}{dt^2}=\omega^2D^2_0u +2\epsilon \omega D_0D_1 u +\epsilon^2 D_1^2 u
\end{equation}
and we look for a small solution with initial data 

$u(0)=\epsilon a_0+o(\epsilon)$ and $\dot u(0)=o(\epsilon)$; we use  the {\it ansatz}
\begin{equation}
  \label{eq:dev-doublech}
  u=\epsilon u_1(T_0,T_1) +\epsilon^2 r(T_0,T_1,\epsilon);
\end{equation}
so we have:
\begin{equation}
  \frac{du}{dt}=\epsilon [ \omega D_0u_1 +\epsilon D_1 u_1 ]+ \epsilon^2 [\omega D_0 r + \epsilon D_1 r ]
\end{equation}
and
\begin{equation}
\label{eq:d2usdt2free}
 \frac{d^2u}{dt^2}=\epsilon \omega^2D^2_0u_1  +\epsilon^2 [2  \omega D_0D_1 u_1 + \omega^2 D_0^2r ] +\epsilon^3 [D_1^2 u_1 +  \mathcal{D}_2 r] 
\end{equation}
with
\begin{equation}
  \mathcal{D}_2 r = \frac{1}{\epsilon}\left( \frac{d^2 r}{d t^2} -\omega^2 D_0^2 r \right)= 2 \omega D_0D_1 r +\epsilon D_1^2 r
\end{equation}
We plug expansions \eqref{eq:dev-doublech},\eqref{eq:d2usdt2free} into \eqref{eq:grosse-cub-lib};  by identifying the powers of
 $\epsilon$ in the expansion of  equation \eqref{eq:grosse-cub-lib}, we obtain:
\begin{align}
\Bigg \{
  \begin{array}[h]{rl}
  &\omega^2(D_0^2 u_1 + u_1)=0 \\
%\label{eq:D02u1+u1=0l}
 & (D_0^2 r + r)=\frac{S_2}{\omega^2} \quad \text{ with } \label{eq:D02r=S2}\\
  \end{array} 
\end{align}
\begin{equation}
S_2= -\frac{1}{\epsilon^2} \left[   c (\epsilon u_1+\epsilon^2 r)^2
  +\frac{d}{\epsilon}(\epsilon u_1+ \epsilon^2 r)^3 \right] -2 \omega D_0D_1
u_1 -\epsilon {\cal R}(u_1,r,\epsilon)
%% \label{eq:S2=}
\end{equation}
where
\begin{equation}
  {\cal R}=D_1^2 u_1 +{\cal D}_2r;
\end{equation}
we can manipulate to obtain:
\begin{equation}
S_2= -\left[   c u_1^2 +d u_1^3 +2 \omega D_0D_1 u_1 + \epsilon R(u_1,r,\epsilon) \right]
 \label{eq:S2=}
\end{equation}
where
\begin{align}
 R(u_1,r,\epsilon)=\left [  {\cal R}+2cu_1r+3du_1^2r+\epsilon\rho(u_1,r,\epsilon) \right]
\end{align}
with a polynomial  $\rho(u_1,r,\epsilon)=c r^2 +3du_1r^2+\epsilon d r^3$ .

We set $\theta(T_0,T_1)= T_0+\beta(T_1)$ noticing  $D_0
\theta=1, \; D_1 \theta=D_1 \beta$; we solve  equation
\eqref{eq:D02r=S2} with:
\begin{equation}
  \label{eq:u1=}
  u_1=a(T_1) \cos(\theta)
\end{equation}
and we obtain
\begin{multline}
%  \label{eq:S2=}
  S_2=  \frac{-c a^2}{2}(1+ \cos(2 \theta))-\frac{d a^3}{4}\left( \cos(3\theta)+3 \cos(\theta) \right ) + \\2\omega(D_1 a \sin(\theta)+ aD_1 \beta \cos(\theta)) - \epsilon  R(u_1,r,\epsilon);
\end{multline}
we gather terms at angular frequency $1$: 
\begin{multline}
  S_2=-\frac{d a^3}{4}3 \cos(\theta)  +2\omega \left [D_1 a \sin(\theta)+ aD_1 \beta \cos(\theta) \right]+S_2^{\sharp} -  \epsilon  R(u_1,r,\epsilon)
\end{multline}
where
\begin{equation}
  S_2^{\sharp}= \frac{-c a^2}{2}(1+ \cos(2 \theta))-\frac{d a^3}{4} \cos(3\theta).
\end{equation}
By imposing
\begin{multline}
  \label{eq:Da,Dbeta-libre}
  D_1a=0 \text{ and } \quad  2\omega aD_1 \beta=3\frac{d a^3}{4},
  \text{ so that } \\
a=a_0,  \quad \beta=\beta_0 T_1 \text{ with } \beta_0= 3\frac{d a^2}{8 \omega} T_1,
\end{multline}
we get that 
$S_2=S_2^{\sharp}-\epsilon R(u_1,r,\epsilon) $ no longer contains   any term  at  frequency $1$.

In order to show that  $r$ is bounded, after eliminating  terms at
angular frequency $1$, we go back to the $t$ variable in the second
 equation \eqref{eq:D02r=S2}.
\begin{align}
&  \ddot{r}+\omega^2r=\frac{\tilde{S_2}}{\omega^2} \quad \text{  with
} \label{eq:ddotr=free1ddl}\\
&\tilde{S_2}=S_2^{\sharp}(t,\epsilon) -  \epsilon  \tilde R(u_1,r,\epsilon) \text{ where } \\
&S_2^{\sharp}(t,\epsilon)=\frac{-c a^2}{2} \left [1+ \cos(2 (\omega t+\beta(\epsilon t)) ) \right] -\frac{d a^3}{4} \cos(3(\omega t+\beta(\epsilon t)))  \\
& \quad \quad =\frac{-c a^2}{2}(1+ \cos(2 (\omega t+\beta_0 \epsilon
t))-\frac{d a^3}{4}\left( \cos(3(\omega t+\beta_0 \epsilon t)) \right
) \\ &\text{ with } 
%\beta_0=\frac{3da^2}{8 \omega} \text{ and } 
\tilde R(u_1,r,\epsilon)= R(u_1,r,\epsilon)- \mathcal{D}_2 r
\end{align}
in which the remainder $\tilde R$ is expressed  with  variable $t$.

\begin{proposition}
 There exists $\gamma>0$ such that for all $t \le t_{\epsilon}=\frac{\gamma}{\epsilon}$, the solution of  \eqref{eq:grosse-cub-lib}, with $u(0)=\epsilon a_0+o(\epsilon), ~\dot u(0)=o(\epsilon)$, satisfies the  following expansion 
 $$u(t)=\epsilon a_0~ \cos( \nu_{\epsilon} t) ~+~ \epsilon^2 r(\epsilon,t)$$
where 
\begin{equation}
  \label{eq:nualpha}
  \nu_{\epsilon}=\omega +3\epsilon\frac{d a^2}{8 {\omega}}
\end{equation}
and   $r$ is uniformly bounded in  $C^{2}(0,t_{\epsilon})$ .
\end{proposition}
\begin{proof}
Let us use   lemma \ref{eq:lemmew } with equation \eqref{eq:ddotr=free1ddl}; set $S=S_2^{\sharp}$;  as we
have enforced  \eqref{eq:Da,Dbeta-libre}, it is a periodic bounded
function  orthogonal to $e^{\pm it}$, it satisfies  lemma hypothesis;
similarly set  $g=\tilde R$; it is a polynomial in variable $r$ with
coefficients which are bounded functions, so it is a lipschitzian
function on
bounded subsets and satisfies lemma  hypothesis.
\end{proof}

\subsection{Forced vibration,  double scale expansion of order 1}

\subsubsection{Derivation of  the expansion}
Here we consider a similar system with a sinusoidal  forcing at a frequency close
to the free frequency (so called primary resonance); in the linear case, without damping, it is well
known that the solution is no longer bounded when the forcing
frequency goes to the free frequency. Here, we consider the
mechanical system of previous section but with periodic  forcing and we
include some light damping term; the scaling of the forcing term
is chosen so that the expansion works properly; this is a known difficulty, for
example see \cite{nayfeh86}.

\begin{equation}
  \label{eq:grosse-cub-forc}
  \ddot u + \omega^2 u +\epsilon \lambda \dot u +cu^2 +\frac{d u^3}{\epsilon}=\epsilon^2 F \cos(\tilde \omega_{\epsilon} t).
\end{equation}
We assume positive damping, $\lambda >0$ and  excitation  frequency $\tilde \omega_{\epsilon}$ is {\it
  close} to an eigenfrequency of the  linear system in the following way:
\begin{equation}
\label{eq:tiledeomeag=}
\tilde \omega_{\epsilon}= \omega +\epsilon \sigma.
\end{equation}
We look for a small solution with a double scale expansion; to
simplify the computations, the fast scale  $T_0$ is chosen $\epsilon$
dependent  and we set:

\begin{equation}
  \label{eq:T0T1D0D1}
  T_0=\tilde \omega_{\epsilon} t, \quad T_1=\epsilon t \text{ and } \quad
  D_0u=\frac{\partial u}{\partial T_0}, ~ D_1u=\frac{\partial u}{\partial T_1};
\end{equation}
so
\begin{equation}
\label{eq:dudt,d2udt2}
  \frac{du}{dt}=\tilde\omega D_0u +\epsilon D_1 u \; \text{ and } \;  \frac{d^2u}{dt^2}=\tilde\omega_{\epsilon}^2D^2_0u +2\epsilon \tilde\omega_{\epsilon} D_0D_1 u +\epsilon^2 D_1^2 u;
\end{equation}
equation \eqref{eq:tiledeomeag=} provides
\begin{equation}
  \label{eq:tiledeomeag2=}
\tilde \omega_{\epsilon}^2= \omega^2 +2\epsilon \omega\sigma+\epsilon^2\sigma^2.
\end{equation}
With \eqref{eq:tiledeomeag=}, \eqref{eq:T0T1D0D1},
\eqref{eq:dudt,d2udt2}, \eqref{eq:tiledeomeag2=} and the  {\it ansatz}
\begin{equation}
  \label{eq:dev-doublech-forc}
  u=\epsilon u_1(T_0,T_1) +\epsilon^2 r(T_0,T_1,\epsilon),
\end{equation}

we obtain:
\begin{gather}
  \frac{du}{dt}=\epsilon \frac{du_1}{dt} + \epsilon^2 \frac{dr}{dt}= \epsilon \frac{du_1}{dt} +  \epsilon^2 \omega D_0 r +\epsilon^2 (\frac{dr}{dt} -\omega D_0 r)=\\
 \epsilon[ \tilde \omega D_0u_1 +\epsilon D_1 u_1] +  \epsilon^2 \omega D_0 r +\epsilon^2 (\frac{dr}{dt} -\omega D_0 r)=\\
 \epsilon[  \omega D_0u_1+ \epsilon \sigma D_0 u_1 +\epsilon D_1 u_1] +  \epsilon^2 \omega D_0 r +\epsilon^2 (\frac{dr}{dt} -\omega D_0 r)
\end{gather}
where we  remark that $\frac{dr}{dt} -\omega D_0 r=\epsilon \sigma D_0 r
+\epsilon D_1 r $ is of degree 1 in $\epsilon$. For the  second derivative, as for the case without forcing, we introduce
\begin{align}
  {\cal D}_2r&= \frac{1}{\epsilon}\left
    (\frac{d^2r}{dt^2}-\omega^2D_0^2r \right) \text{ with the expansion}\\
  {\cal D}_2r&=2 \omega [ \sigma D_0^2 r +  D_0D_1 r ]+ \epsilon\left[ \sigma^2 D_0^2r +2 \sigma D_0D_1r +D_1^2r \right];
\end{align}
\begin{align}
  \frac{d^2u}{dt^2}&=\epsilon\frac{d^2u_1}{dt^2}+\epsilon^2\frac{d^2r}{dt^2}=
\epsilon\frac{d^2u_1}{dt^2}+\epsilon^2 \omega^2D_0^2r +\epsilon^3 {\cal D}_2 r\\
& =\epsilon \left [ \tilde\omega^2D^2_0u_1 +2\epsilon \tilde\omega D_0D_1 u_1 +\epsilon^2  D_1^2 u_1 \right ] \\ 
&\qquad \qquad \qquad +\epsilon^2 \omega^2D_0^2r +\epsilon^3 {\cal D}_2 r\\
& =\epsilon \big \{ \omega^2D^2_0u_1+ 2\epsilon \omega \left(   \sigma D_0^2u_1 +   D_0D_1 u_1 \right)+\\&
\qquad\qquad\qquad \epsilon^2 \left[ \sigma^2D_0^2u_1+2  \sigma D_0D_1 u_1 + D_1^2 u_1\right ] \big \} \\ 
&\qquad\qquad\qquad\qquad\qquad+\epsilon^2 \omega^2D_0^2r +\epsilon^3 {\cal D}_2 r
\end{align}

 the last term in the right hand side  will be part of the remainder  $R$ of equation \eqref{eq:D0^2r+r=S2/w2}.
We plug previous expansions into \eqref{eq:grosse-cub-forc};
we obtain:
\begin{gather}
\label{eq:D02u1+u1=0f}
\Bigg \{  
\begin{array}[]{rl}
&\omega^2 (D_0^2 u_1 + u_1)=0 \\
&  D_0^2 r + r=\frac{S_2}{\omega^2} \quad \text{ with}
\end{array}
\end{gather}

\begin{gather}
 \label{eq:D0^2r+r=S2/w2}
S_2= -\left \{ c u_1^2 +d u_1^3 +2  \omega [D_0D_1 u_1 +\sigma D_0^2u_1]
+\lambda \omega D_0 u_1 \right \} \\
+F \cos(T_0)- \epsilon R(u_1,r,\epsilon) 
\end{gather}
 \text{and with}
\begin{gather}
R(u_1,r,\epsilon)= D_1^2 u_1+ 2cu_1 r +3 d u_1^2 r+ 
\sigma^2 D_0^2u_1 +2 \sigma D_0D_1u_1 +\\
\lambda(\omega D_0r +\sigma D_0u_1 +D_1u_1) 
+{\cal D}_2 r\\
+\lambda(\frac{dr}{dt}-\omega D_0r) +\epsilon \rho(u_1,r,\epsilon). 
\end{gather}
Set $\theta(T_0,T_1)= T_0+\beta(T_1)$. We solve the first equation of \eqref{eq:D02u1+u1=0f}   :
\begin{equation}
  \label{eq:u1forc=}
  u_1=a(T_1) \cos(\theta);
\end{equation}
then we use $ T_0  =\theta(T_0,T_1) -\beta(T_1)$
and we obtain

\begin{multline}
  S_2=   \frac{-c a^2}{2}(1+ \cos(2 \theta))-\frac{d a^3}{4}\left( \cos(3\theta)+3 \cos(\theta) \right ) + 
\\2\omega(D_1 a \sin(\theta)+ aD_1 \beta \cos(\theta)) +2 \sigma \omega a \cos(\theta) +  a \lambda \omega \sin(\theta)  \\
+F \sin(\theta) \sin( \beta(T_1)) +F \cos(\theta) \cos(\beta(T_1)) -
\epsilon  R(u_1,r,\epsilon) 
\end{multline}
or
\begin{multline}
  S_2= \left[ 2\omega D_1 a +\lambda a \omega +F\sin(\beta)   \right ]\sin(\theta) \\
+\left [ 2\omega a D_1 \beta +2\sigma\omega a-\frac{3da^3}{4} + F \cos (\beta) \right] \cos(\theta) \\
+  S_2^{\sharp} ~  \quad
-\epsilon  R(u_1,r,\epsilon)
\end{multline}
with
\begin{equation}
  S_2^{\sharp}=   \frac{-c a^2}{2}(1+ \cos(2 \theta))-\frac{d a^3}{4}\left( \cos(3\theta) \right );
\end{equation}

note that  $S_2^{\sharp}$ is a periodic function with frequency
strictly multiple of $1$.
\paragraph{Orientation.}
By enforcing 
\begin{align}
\label{eq:D1aD1beta1} 
\Bigg \{
\begin{array}[h]{rl}
&  2 \omega D_1 a +\lambda a \omega =-F \sin(\beta ) \\
%\label{eq:D1aD1beta2}
& 2 a \omega D_1 \beta +2\sigma\omega a -\frac{3d a^3}{4} =-F\cos(\beta) 
\end{array}
\end{align}

$S_2 =S_2^{\sharp} - \epsilon  R(u_1,r,\epsilon)$ contains
neither term at frequency $1$ nor at a frequency which goes to $1$; this
point will enable to justify this expansion under some  conditions;
before, we study
 stationary solution of this  system and the  stability of the dynamic
 solution in a neighborhood of the  stationary solution.
%le lemme \ref{eq:lemmew } s'applique

\subsubsection{Stationary solution and  stability}

Let us consider the  stationary solution of  \eqref{eq:D1aD1beta1},
%\ref{eq:D1aD1beta2}, 
it satisfies:
\begin{align}
\label{eq:D1aD1beta1stat}
\Bigg \{
  \begin{array}[h]{rl}
\lambda a \omega &+F \sin(\beta) =0\\
 \left (2\omega\sigma -\frac{3d a^2}{4} \right) &+\frac{F\cos(\beta)}{a}=0.
  \end{array}
\end{align}
Now, we study the stability of the solution of \eqref{eq:D1aD1beta1},
%\ref{eq:D1aD1beta2} 
in a neighborhood of this  stationary solution noted $(\bar a, \bar
\beta)$; set $a=\bar a + \tilde a, \beta=\bar \beta +\tilde \beta$,
the linearized  system is written 
$$\binom{D_1 \tilde a}{ D_1 \tilde \beta}=J\binom{\tilde a}{\tilde \beta};$$
manipulating, we obtain the jacobian matrix. 
\begin{gather}
J=  \begin{pmatrix}
    -\frac{\lambda}{2 } & -\frac{F}{2 \omega} \cos(\bar \beta)\\
\frac{9d \bar a}{8 \omega} -\frac{\sigma}{\bar a}& \frac{F}{2 \omega \bar a} \sin(\bar \beta)
%\frac{3d \bar a}{4 \omega} +\frac{F}{2 \omega \bar a^2}\cos(\bar \beta)& \frac{F}{2 \omega \bar a} \sin(\bar \beta)
  \end{pmatrix}
 = \begin{pmatrix}
     -\frac{\lambda}{2 } & a(\sigma -\frac{3d \bar a^2}{8 \omega})\\
\frac{9d \bar a}{8 \omega} -\frac{\sigma}{\bar a}&  -\frac{\lambda}{2 }
  \end{pmatrix}.
\end{gather}
The matrix   trace is $-\lambda$, and the  determinant is 
$$ det(J)=\frac{\lambda^2}{4}-(\frac{9 d \bar a^2}{8 \omega}-\sigma)(\sigma-\frac{3 d \bar a^2}{8 \omega};)$$
we notice that the determinant is strictly positive for $\sigma=0$ so
by continuity, it remains positive for $\sigma$ small; moreover $
\frac{d}{d\sigma}det(J)<0$ for $\sigma<0$ so $det(J)>0$ for $\sigma<0$; by studying the trinomial in $\sigma$, we notice that 
 the determinant is positive when this semi-implicit inequality is satisfied:
 $\sigma \le \frac{3 d \bar a^2}{4 \omega}
 -\frac{1}{2}\sqrt{\frac{9d^2\bar a^4}{16 \omega^2}-\lambda^2}$; 
so in these conditions, the two eigenvalues are negative;
then the  solution of the  linearized system goes to zero; with the
theorem of Poincar\'e-Lyapunov (look in the appendix for  the theorem
\ref{th:poinc-lyapu},)
when the initial data is close enough to the 
  stationary solution, the solution of the system \eqref{eq:D1aD1beta1},
%\ref{eq:D1aD1beta2} 
goes to the stationary solution.
We expand this point, set

\begin{equation}
    y=\binom{a}{\beta}, \quad G(y)= \left( \begin{array}[h]{rl}
-\lambda a \omega &-F \sin(\beta) \\
- \left (2\omega\sigma -\frac{3d a^2}{4} \right) &-\frac{F\cos(\beta)}{a}
  \end{array}
\right);
\end{equation}
the system \eqref{eq:D1aD1beta1stat} may be written $\dot y=G(y)$; denote
 $\bar y= \binom{\bar a}{\bar \beta}$ the solution of
 \eqref{eq:D1aD1beta1stat}; perform the  change of variable $y=\bar y
 +x$, we have $G(\bar y +x)=G(\bar y) +Jx +g(x)
$ with $ g( x) =o(\| x\|$; the theorem \ref{th:poinc-lyapu} may be
applied with $A=J, \; B=0 $, here the function $g$ does not depends  on time.

\begin{proposition}
\label{prop:stab1ddl}
  If  $\sigma \le \frac{3 d \bar a^2}{4 \omega} -\frac{1}{2}\sqrt{\frac{9d^2\bar a^4}{16 \omega^2}-\lambda^2}$
, the stationary solution of  \eqref{eq:D1aD1beta1}
%\ref{eq:D1aD1beta2} 
is  stable in the sense of Lyapunov (if the dynamic solution starts
 close to the
stationary solution of \eqref{eq:D1aD1beta1stat}, it remains close to it and converges to it ); to the  stationary case corresponds the approximate solution of
\eqref{eq:grosse-cub-forc} $u_1=\bar a \cos(T_0+ \bar \beta)$, it is
periodic; for an initial data  close enough  to  this stationary solution, $u_1= a(T_1) \cos(T_0+  \beta(T_1))$ with $a,\beta$ solutions of  \eqref{eq:D1aD1beta1}%,\ref{eq:D1aD1beta2} 
; it goes to the solution \eqref{eq:D1aD1beta1stat} $\bar a, \bar \beta$ when $T_1 \longrightarrow +\infty$.
\end{proposition}
With this result of stability, we can  state precisely the approximation of the solution of \eqref{eq:grosse-cub-forc} by the function $u_1$.

\subsubsection{Convergence of the expansion}
\begin{proposition}
\label{prop:conv-dev-forc}
  Consider the solution of \eqref{eq:grosse-cub-forc} with $$u(0)=\epsilon  a_0+o(\epsilon), ~\dot u(0)=-\epsilon \omega a_0 \sin(\beta_0)+o(\epsilon),$$ with $a_0, \beta_0$ close of the stationary solution $(\bar a, \bar \beta)$, 
$$|a_0-\bar a|\le \epsilon C_1,|\beta_0 -\bar \beta| \le \epsilon C_2;$$
 When  $\sigma \le \frac{3 d \bar a^2}{4 \omega} -\frac{1}{2}\sqrt{\frac{3d\bar a^2}{2 \omega}-\lambda^2}$,
 there exists $\gamma>0$ such that for all $t \leq t_{\epsilon}=\frac{\gamma}{\epsilon}$,
 the following expansion is satisfied
 $$u(t)=\epsilon a(\epsilon t)~ \cos(\tilde \omega_{\epsilon} t+\beta(\epsilon t)) ~+~ \epsilon^2 r(\epsilon,t)$$ with 
  $\omega_{\epsilon}=\omega+\epsilon \sigma$ and $r$ uniformly bounded  in $C^{2}(0,t_{\epsilon})$ and   with $a, ~\beta$ solution of  \eqref{eq:D1aD1beta1}.
%,\ref{eq:D1aD1beta2}
\end{proposition}
\begin{proof}
Indeed after  eliminating
 terms at frequency $1$, we go back to the variable $t$ for  the second equation \eqref{eq:D02u1+u1=0f}
%\ref{eq:D0^2r+r=S2/w2}.
\begin{align}
&  \ddot{r}+\omega^2r=\frac{\tilde{S_2}}{\omega^2} \text{  with }\\
&\tilde{S_2}=S_2^{\sharp}(t,\epsilon) -  \epsilon  \tilde R(u_1,r,\epsilon)
\end{align}
where
\begin{equation}
  \tilde R(u_1,r,\epsilon)=R(u_1,r,\epsilon) 
%-\frac{1}{\epsilon} \left ( \frac{d^2 r}{dt^2}-\omega^2D_0^2r \right)
-{\cal D}_2r -\lambda(\frac{dr}{dt}-\omega D_0 r)
\end{equation}
with all the terms expressed with the variable $t$;
we have 
\begin{equation}
    S_2^{\sharp}(t,\epsilon)=   \frac{-c a^2(\epsilon t)}{2}(1+ \cos(2
    (\tilde \omega_{\epsilon} t+\beta(\epsilon t)))-\frac{d a^3(\epsilon t)}{4}\left(
      \cos(3(\tilde \omega_{\epsilon} t +\beta(\epsilon t) \right );
\end{equation}
this function is {\it} not periodic but is  {\it close} of the periodic function:
\begin{equation}
    S_2^{\natural}(t,\epsilon)=   \frac{-c \bar a^2}{2}(1+ \cos(2
    (\tilde \omega_{\epsilon} t+ \bar \beta ))-\frac{d a^3}{4}\left(
      \cos(3(\tilde \omega_{\epsilon}
      t +\bar \beta ) \right ) 
\end{equation}

and for $t\le \frac{\gamma}{\epsilon}$ as  the solution of  \eqref{eq:D1aD1beta1} is stable:
it remains close to the stationary solution
\begin{equation}
 |a(\epsilon t) -\bar a| \le \epsilon C_1, \quad  |\beta(\epsilon t)-\bar \beta|\le \epsilon C_2
\end{equation}
and
\begin{equation}
  | S_2^{\sharp}- S_2^{\natural}| \le \epsilon C_3;
\end{equation}
so this difference may be included in the remainder
 $\tilde R$.
We use lemma \ref{eq:lemmew } with $S =S_2^{\natural}$; it  satisfies
lemma  hypothesis; similarly, we use $g=\tilde R$; 
it satisfies the hypothesis because it is a  polynomial  in the
variables $r,u_1, \epsilon$, with coefficients which are bounded
functions, so it is  lipschitzian on  bounded subsets.
\end{proof}

\subsubsection{Maximum of the stationary solution, primary resonance}
Consider the stationary solution of \eqref{eq:D1aD1beta1}, it satisfies

\begin{align}
\Bigg \{
  \begin{array}[h]{rl}
\lambda a \omega &=-F \sin(\beta) \\
a \left (2\omega\sigma -\frac{3d a^2}{4} \right) &=-F\cos(\beta)
\end{array}
\end{align}
 manipulating, we get that $a$ is solution of the equation:
\begin{equation}
  f(a,\sigma)=\lambda^2 a^2 \omega^2 +a^2 \left (2\omega \sigma-\frac{3d a^2}{4} \right)^2-F^2=0.
\end{equation}
We compute
\begin{gather}
  \frac{\partial f}{\partial \sigma}=4 a^2 \omega (2 \omega \sigma -\frac{3d a^2}{4} )\\
\frac{\partial f}{\partial a}= 2a \lambda^2\omega^2 +2 a \left (2\omega \sigma-\frac{3d a^2}{4} \right)^2  - 6\frac{da^3}{4}\left (2\omega \sigma-\frac{3d a^2}{4} \right)
\\
%  \frac{\partial f}{\partial a}= 2a \lambda^2\omega^2 +8 a \omega^4\sigma^2-12 da^3\omega^2\sigma+\frac{27d^2a^5}{8}
 \frac{\partial^2 f}{\partial \sigma^2}=8a^2 \omega^2\\
\end{gather}
For  $\sigma$ close enough to the solution of $\frac{\partial f}{\partial \sigma}=0$,
$  \frac{\partial f}{\partial \sigma}$ is small, $\frac{\partial
  f}{\partial a}$ is not zero, and with the implicit function theorem this equation
  defines a function  $a(\sigma)$; lets use :
$$\frac{\partial a}{\partial \sigma}=-\frac{\frac{\partial f}{\partial \sigma}}{\frac{\partial f}{\partial a}}
\text{ and } \quad \frac{\partial^2 a}{\partial \sigma^2}=-\frac{\frac{\partial^2 f}{\partial \sigma^2}}{\frac{\partial f}{\partial a} }.$$
In our case, when

  $\frac{\partial a}{\partial \sigma}=0$, we have
\begin{equation}
  \sigma=\frac{3d a^2}{8 \omega}, \quad  \frac{\partial f}{\partial a}=2a \lambda^2 \omega^2, \quad
 \frac{\partial^2 f}{\partial \sigma^2}=8 a^2\omega^4, \;
\end{equation}
so the second derivative $\frac{\partial^2 a}{\partial \sigma^2}<0$
and $a$ is maximum at the frequency of the free periodic solution.

\begin{proposition}
  The stationary solution of \eqref{eq:D1aD1beta1}
%,\ref{eq:D1aD1beta2} 
satisfies 
\begin{align}
\label{eq:statsol1ddl}
\Bigg \{
  \begin{array}[h]{rl}
  \lambda a \omega &+F \sin(\beta)=0 \\
2a \omega\sigma -\frac{3d a^3}{4} &+F\cos(\beta)=0
\end{array}
\end{align} 
it reaches its maximum amplitude for $\sigma=\frac{3d a^2}{8\omega} $
and $\beta=\frac{\pi}{2}+ k \pi$; the excitation is at the angular frequency 
$$\tilde \omega_{\epsilon}= \omega+3\epsilon \frac{d a^2}{8 \omega} +O(\epsilon^2)
~ \text{ and } ~ F=\lambda \omega a$$
it is the angular frequency $\nu_{\epsilon}$ of  the free periodic solution  \eqref{eq:nualpha}
for this frequency, the approximation (of the  solution up to the order $\epsilon$)  is periodic:
\begin{equation}
  u(t)=\epsilon\frac{F}{\lambda \omega}\sin(\tilde \omega_{\epsilon} t)+\epsilon^2 r(\epsilon,t)
\end{equation}

\end{proposition}
\begin{rem}
  We remark that this value of $\sigma=\frac{3d a^2}{8\omega} $ is
  indeed smaller than  the maximal value that  $\sigma$ may reach 
  in order that the previous expansion  converges as
indicated in  proposition \ref{prop:conv-dev-forc}.
\end{rem}
\begin{rem}
  We note also that when the stationary solution reaches its maximum
  amplitude we have $ F=\lambda \omega a$ and so we can recover the
  damping ratio $\lambda$ from such a forced vibration experiment;
  this is a close link with the linear case (see for example
  \cite{geradin-rixen} or the English translation
  \cite{geradin-rixen-eng}). This is quite interesting in practice as
  the damping ratio is usually difficult to measure; we have here a
  kind of stability result for this experiment.
\end{rem}

\subsubsection{Computation of stationary solution}
\begin{figure}[h] 
\includegraphics[height=7cm, width=10cm]{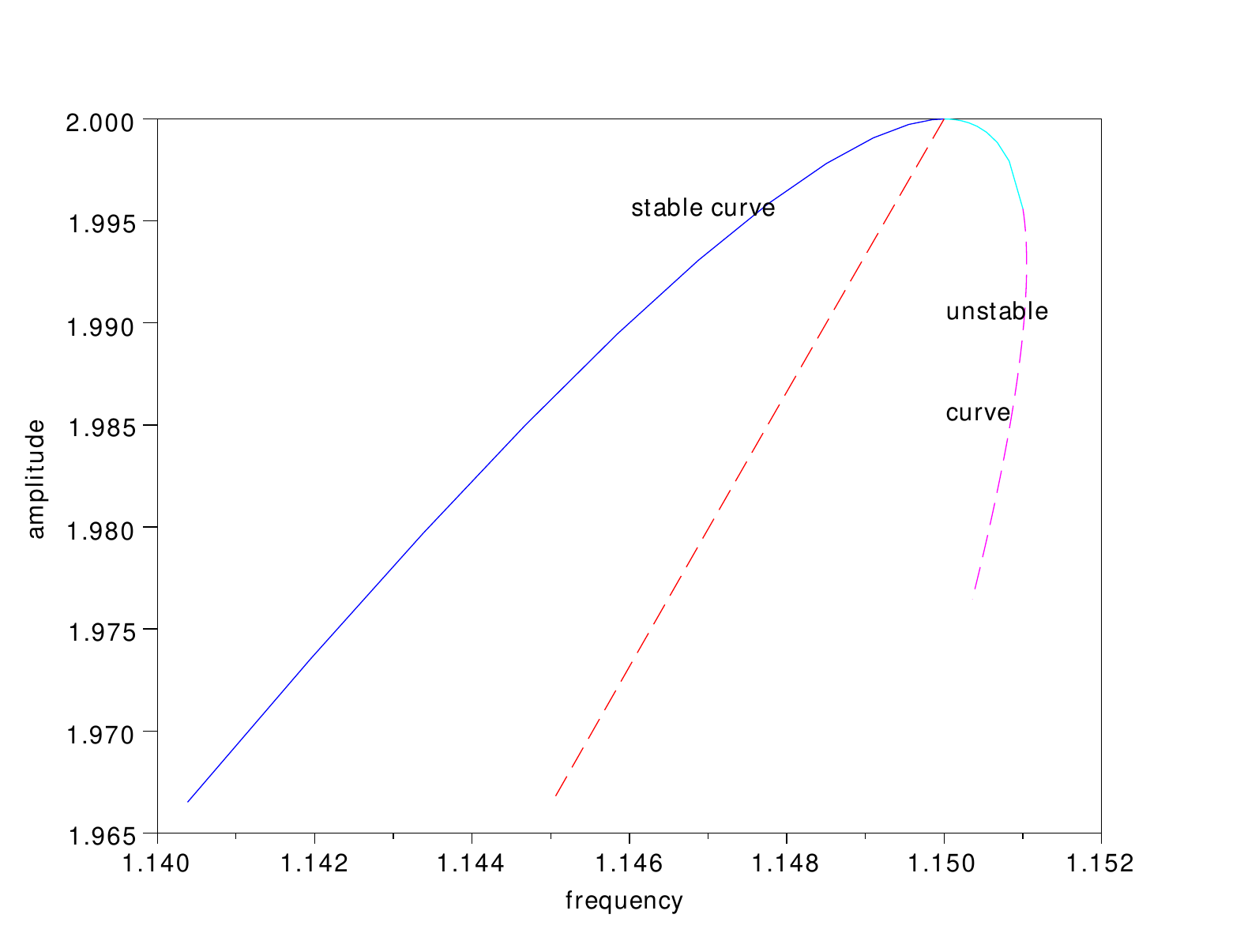}
\caption{amplitude versus frequency of stationary forced solution in
  blue and magenta;
  amplitude of
  free solution in red}
\label{fig:ampli-freq1ddl}
\end{figure}

We have numerically solved equation \eqref{eq:statsol1ddl} for a range of sigma
around the value $\sigma=\frac{3d a^2}{8\omega} $ for which the
amplitude is maximum; we have chosen $\epsilon=.1; \lambda=1/2; F=1; \omega=1;
d=1$; in figure \ref{fig:ampli-freq1ddl}, the solid line shows the solution of this equation that we have
solved with several values of sigma using the routine \texttt{fsolve}
of \texttt{Scilab} which implements a modification of the Powell
hybrid method. We have noticed in proposition \ref{prop:stab1ddl} that
the solution is stable when $\sigma$ is not too large; indeed the
routine \texttt{fsolve} fails to solve  the equation when we
increase too much $\sigma$; to go further this point, with the same routine, we have computed various
values of sigma for decreasing values of the amplitude; we have
plotted this solution with a magenta dotted line.
We have added a red  dotted line which is the amplitude of the free
undamped solution and we notice that it crosses the stationary
solution at the point where it reaches its maximum value as stated in
previous proposition  \ref{prop:stab1ddl}. 

\subsubsection{Dynamic solution}
\begin{figure}[htbp] 
\centering
\begin{minipage}[t]{0.4\linewidth}
\centering
\includegraphics[height=7cm, width=7cm]{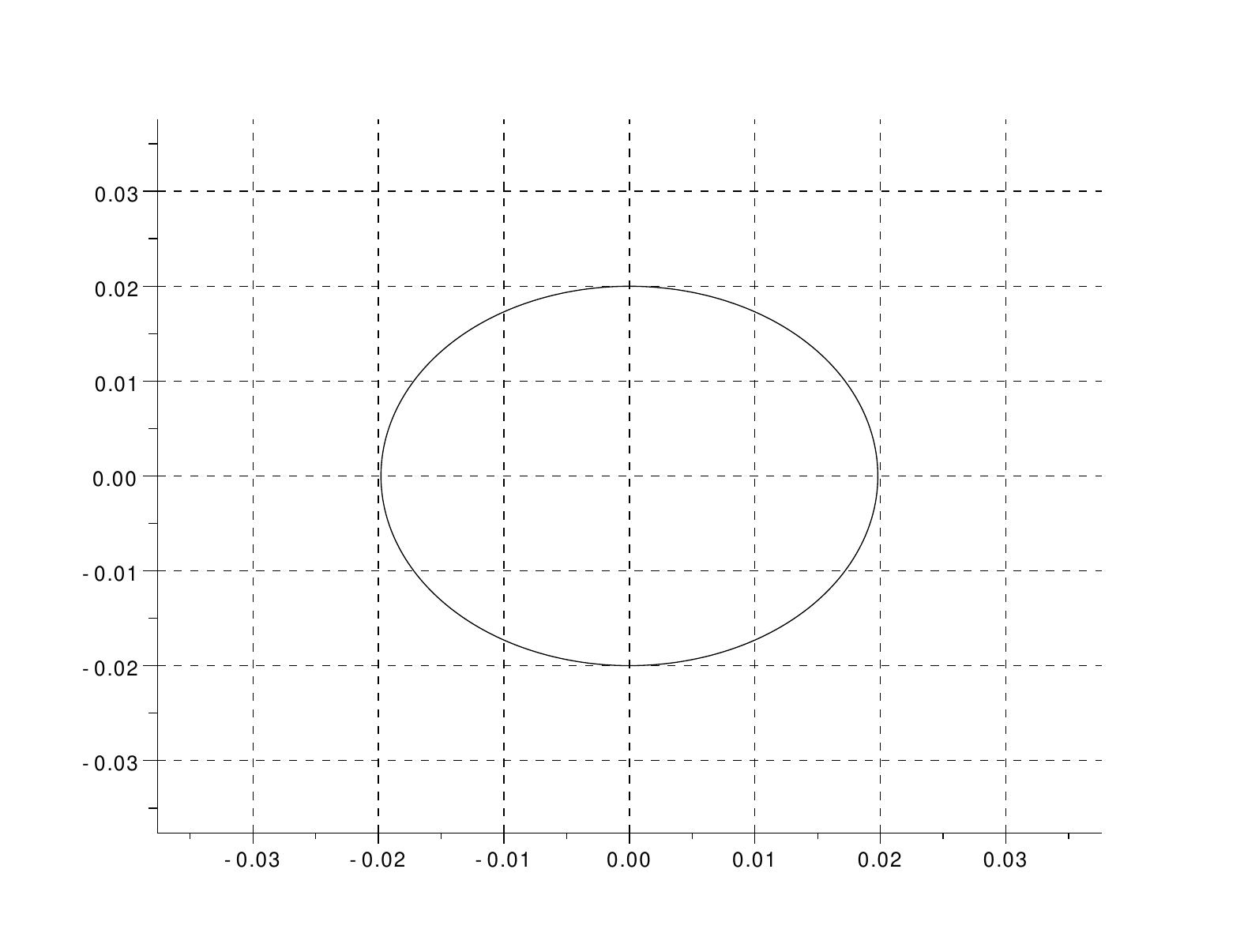}
\caption{Phase portrait for $u_0=0.019796915, \omega_{\epsilon}= 1.0143379$}
\label{fig:soldynu.0197omf1.014}
\end{minipage}
%\end{figure}
%\begin{figure}[p] 
\begin{minipage}[t]{0.4\linewidth}
\centering
\includegraphics[height=7cm, width=4cm]{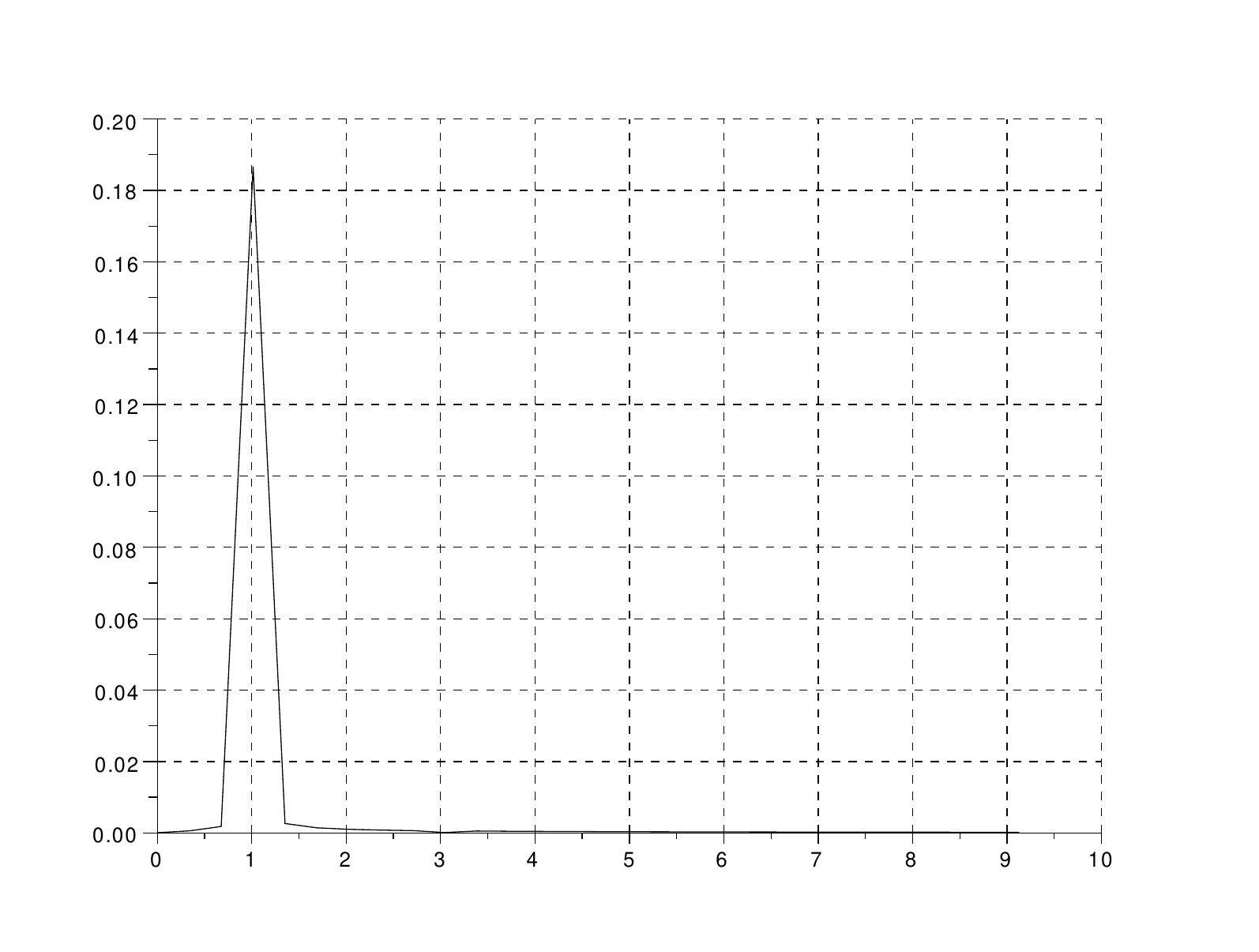}
\caption{Absolute value of the Fourier transform for $u_0=0.019796915, \omega_{\epsilon}= 1.0143379$}
\label{fig:four-soldynu.0197omf1.014}
\end{minipage}
\end{figure}

\begin{figure}[htbp] 
\includegraphics[height=9cm, width=12cm]{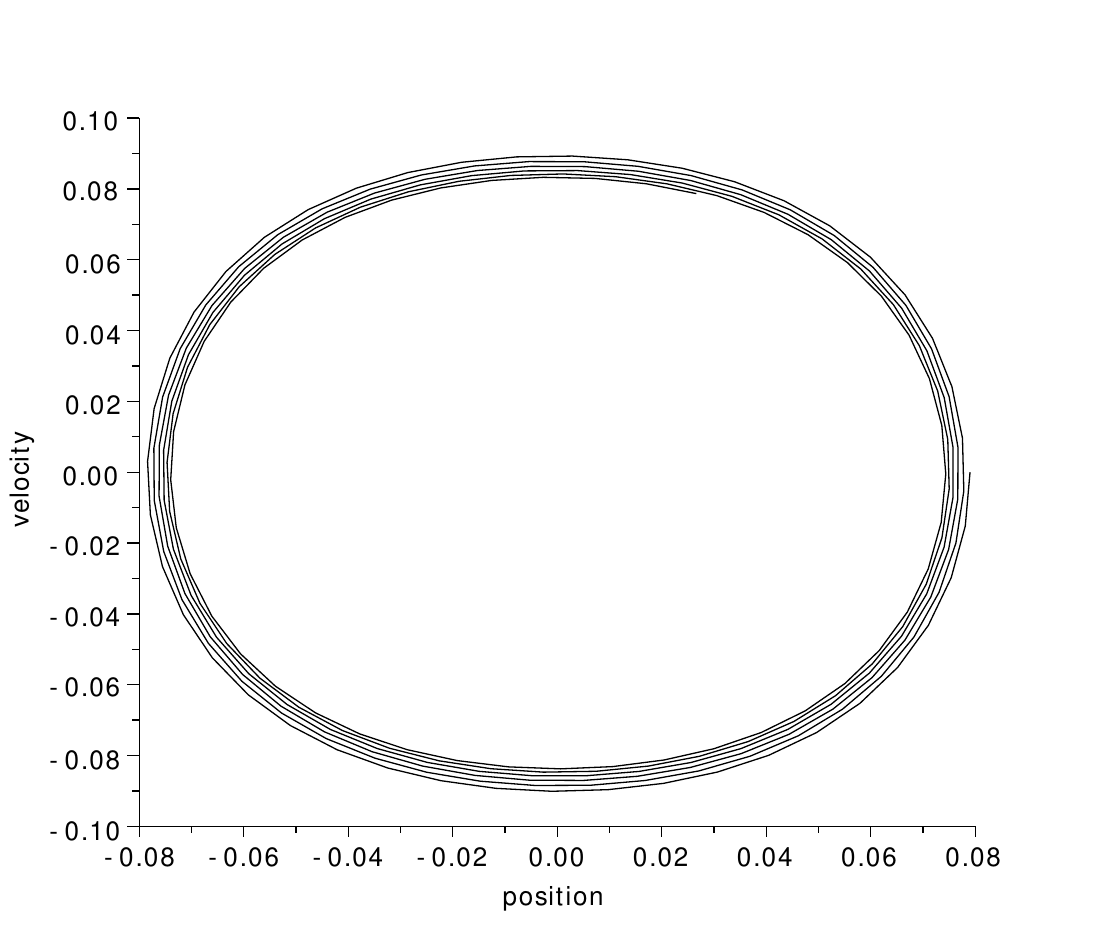}
\caption{Phase portrait for $u_0=0.079, \omega_{\epsilon}= 1.0143379$}
\label{fig:soldynu.079omf1.014}
\end{figure}
\begin{figure}[htbp] 
\includegraphics[height=7cm, width=9cm]{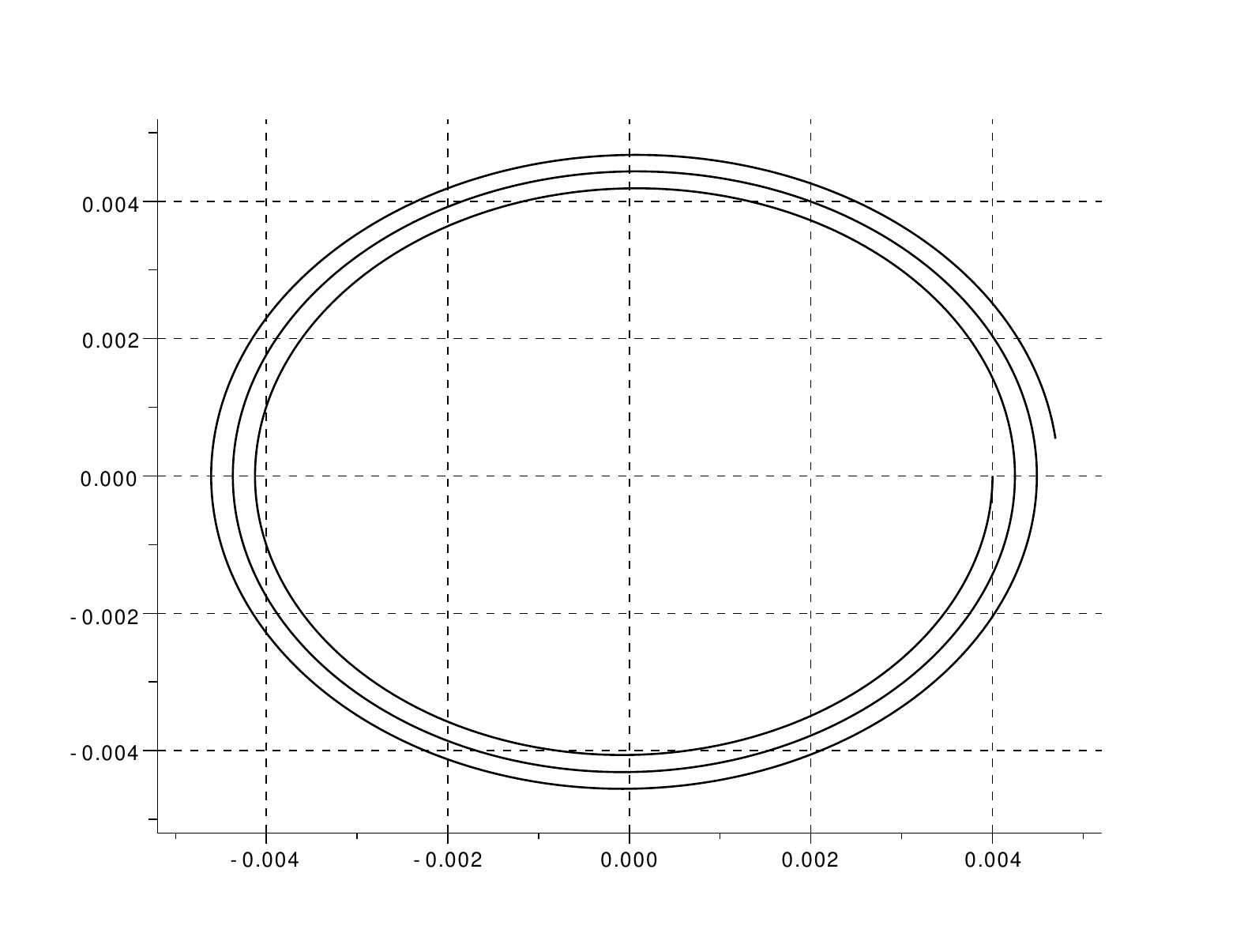}
\caption{Phase portrait for $u_0=0.004, \omega_{\epsilon}= 1.0143379$}
\label{fig:soldynu.004omf1.014}
\end{figure}
\begin{figure}[htbp] 
\includegraphics[height=12cm, width=19cm]{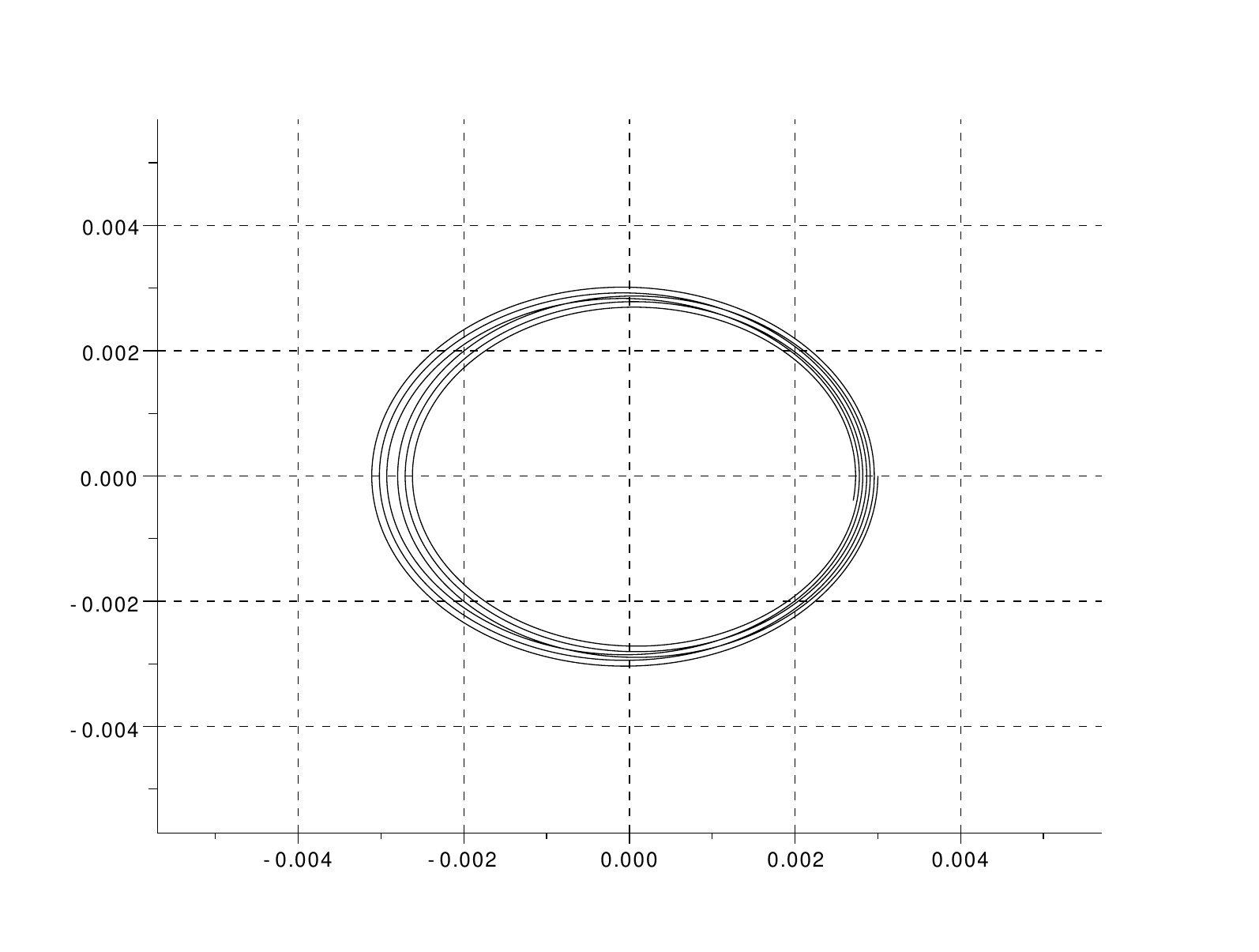}
\caption{Phase portrait for $u_0=0.003, \omega_{\epsilon}= 0.5$}
\label{fig:soldynu.003omf0.5}
\end{figure}

\begin{figure}[htbp] 
\centering
\begin{minipage}[t]{0.4\linewidth}
\centering
\includegraphics[height=7cm, width=10cm]{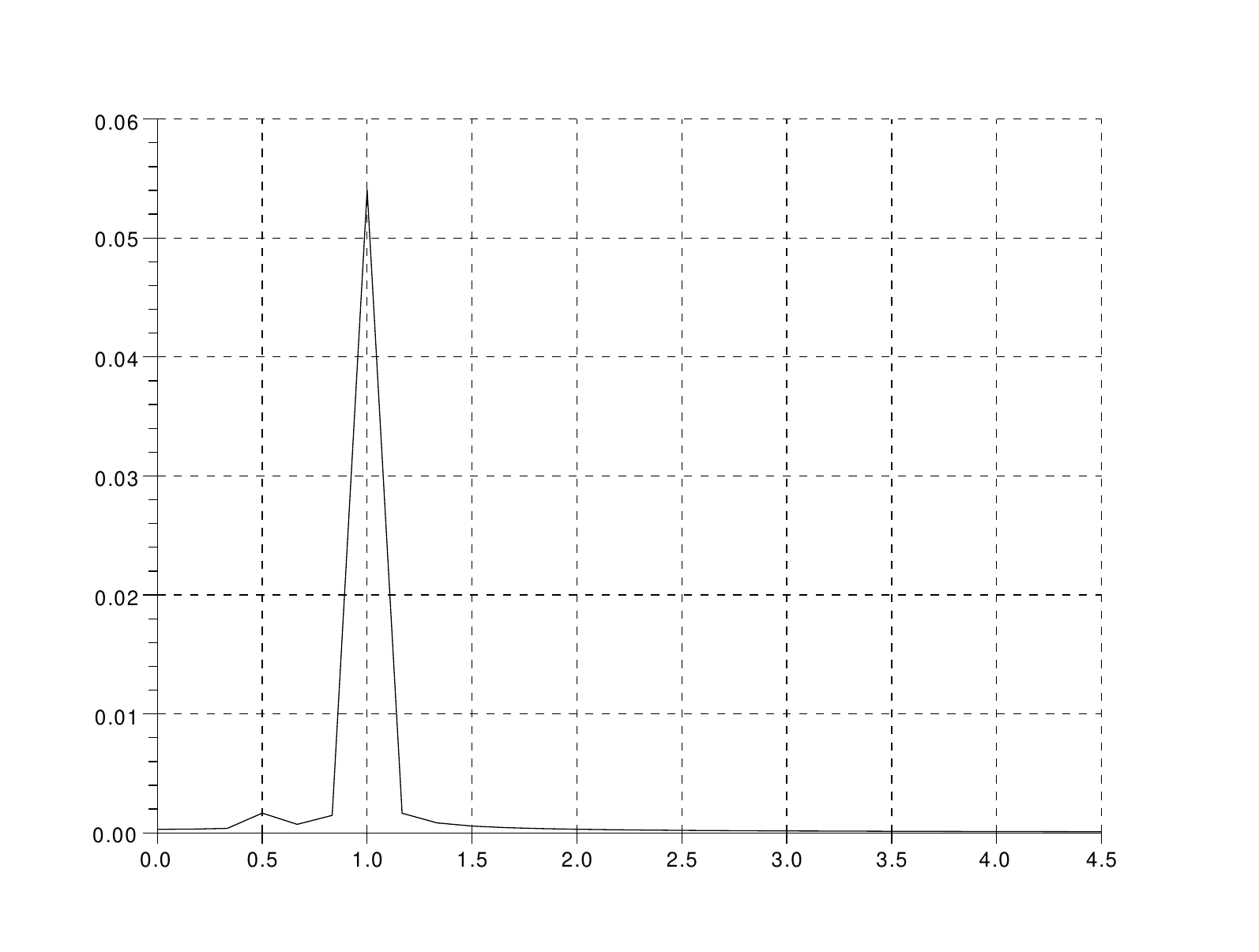}
\caption{Fourier transform for $u_0=0.003, \omega_{\epsilon}= 0.5$}
\label{fig:four-soldynu.003omf0.5}
\end{minipage}
%\end{figure}
%\begin{figure}[p] 
\begin{minipage}[t]{0.4\linewidth}
\centering
\includegraphics[height=7cm, width=10cm]{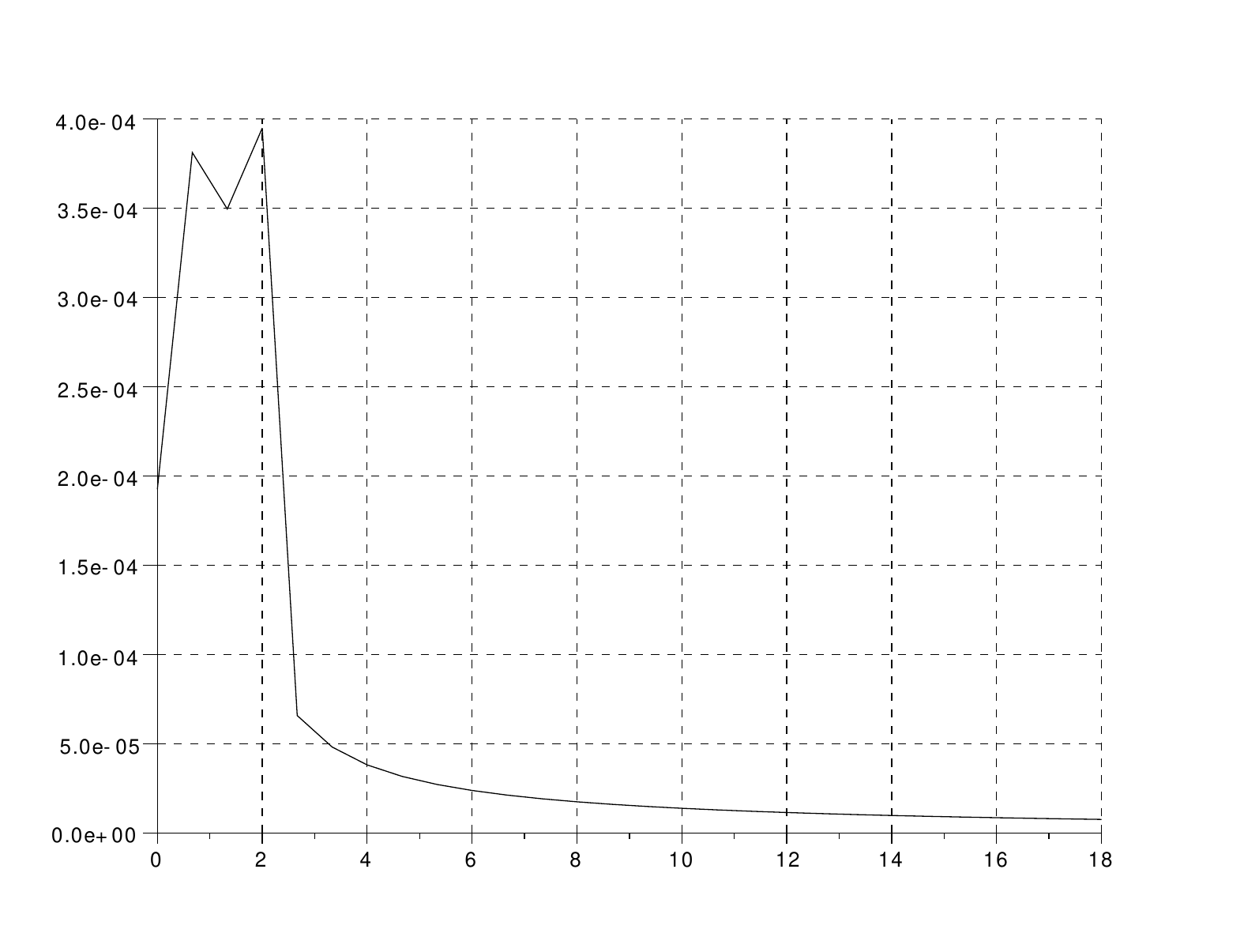}
\caption{Fourier transform for $u_0=0.0001, \omega_{\epsilon}= 2$}
\label{fig:four-soldynu.0001omf2}
\end{minipage}
\end{figure}

\begin{figure}[htbp] 
\includegraphics[height=7cm, width=10cm]{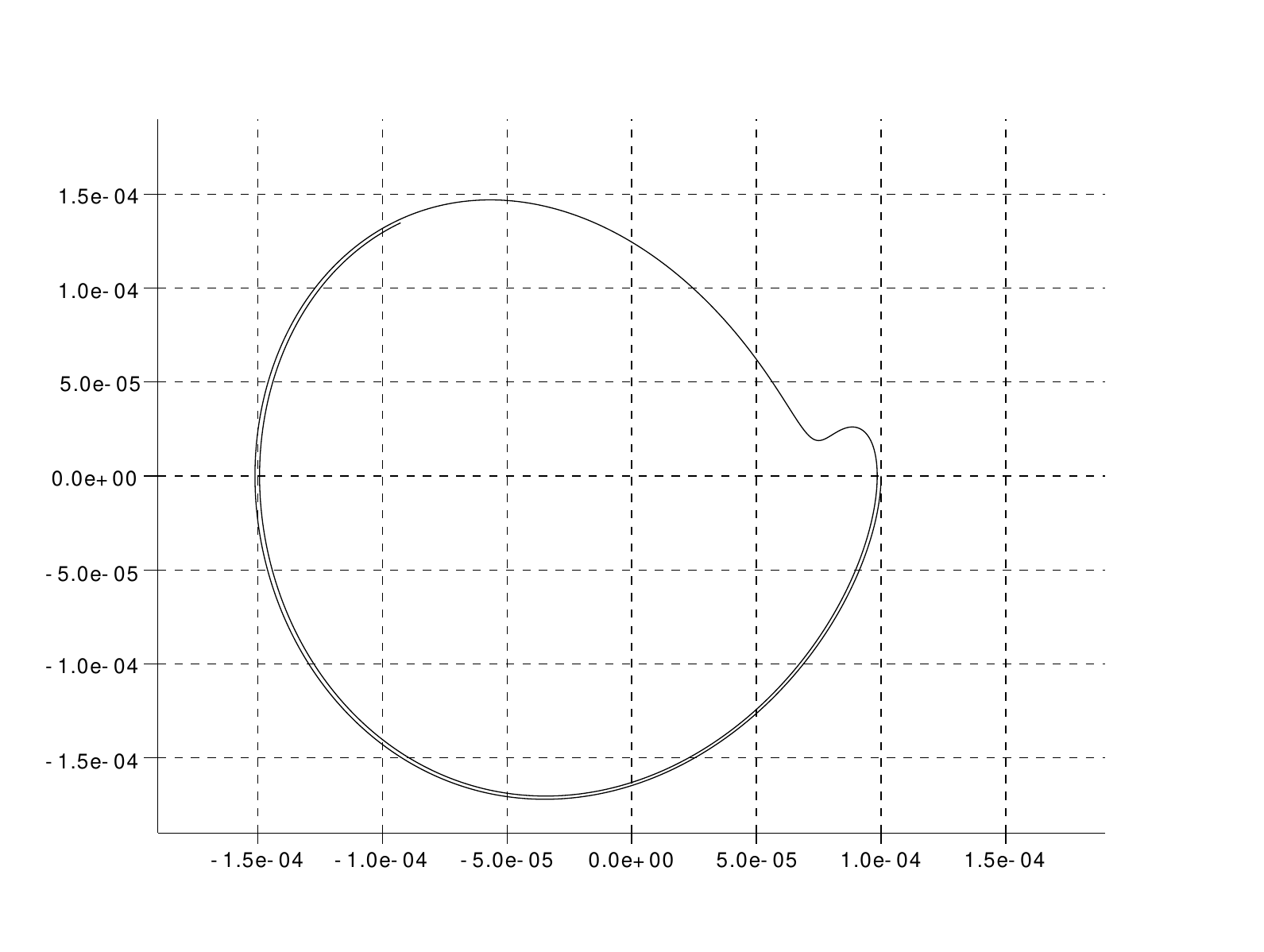}
\caption{Phase portrait for $u_0=0.0001, \omega_{\epsilon}= 2$}
\label{fig:soldynu.0001omf2}
\end{figure}

For various values of the initial condition, we compute numerically
the solution of \eqref{eq:grosse-cub-forc} with a standard theta method.
We use $\epsilon=.01, \lambda=1/2, F=1; \omega=1, d=1$

In figure \ref{fig:soldynu.0197omf1.014}, we find the phase portrait of the solution  with initial values 
$u(0)=0.019796915  , u'(0)=0$  
so that the
angular frequency of the applied force is $\tilde \omega_{\epsilon}= 1.0143379$ , we notice that the solution looks periodic (up to the numerical approximation of the
method); the initial value of the displacement is computed from a
value of $a, \sigma$ of the stationary solution \eqref{eq:statsol1ddl}
which is computed in the
previous paragraph. The Fourier transform in  figure
\ref{fig:four-soldynu.0197omf1.014} shows only one peak at the angular frequency
$1.0143379$ which is the angular frequency of the applied force.

In figure \ref{fig:soldynu.079omf1.014} , for the same value of the frequency of the applied
force, we find the phase portrait of the solution  with initial values 
$u(0)=0.079  , u'(0)=0$; the initial value is  larger than the one
of the stationary solution and we notice that the solution is
decreasing as expected from the stability of the stationary solution.

We find an analogous behavior with an initial value smaller than the
stationary solution; in figure  \ref{fig:soldynu.004omf1.014}, for the same value of the frequency of the applied
force, we find the phase portrait of the solution  with initial values 
$u(0)=0.004  , u'(0)=0$; here the solution is increasing  as expected from the stability of the stationary solution.
 
In the case where $a, \omega_{\epsilon}$ are far from the stationary
curve, we suspect that the frequency content of the response will
involve the frequency of the applied force and some frequency due to
the system; in figure \ref{fig:soldynu.003omf0.5}, we find for $a=0.3,\omega_{\epsilon}=0.5$ the phase portrait of the solution, the
frequency transform is in figure \ref{fig:four-soldynu.003omf0.5}; on this plot of
the Fourier transform, we notice two peaks icluding the angular
frequency $\omega_{\epsilon}=0.5$ of the applied load. 

For large values of  $\omega_{\epsilon}$, the phase portrait is less
regular see figure \ref{fig:soldynu.0001omf2} for $a=0.01,\omega_{\epsilon}=2$ ; we find
also two peaks for the Fourier transform in figure \ref{fig:four-soldynu.0001omf2}.

\clearpage
\section{System with a strong local cubic non linearity}
In the previous  section, we have derived a double scale expansion of a
solution of a one degree of freedom free  vibrations system and damped
vibrations with sinusoidal forcing with frequency close  to  free
vibration frequency.
Now, we extend the results to the case of multiple degrees of freedom.
\subsection{Free vibrations, double scale expansion}
\label{subsec:syst-free-vib}

We consider a system of vibrating masses attached to springs:
\begin{equation}
  \label{eq:system}
  M\ddot u + K u +\Phi(u,\epsilon)=0.
\end{equation}
The mass matrix $M$ and the rigidity matrix $K$ are assumed to be
symmetric and positive definite. See an example in section \ref{subsub:numericals}
We assume that the non linearity is local, all components are zero except for two components $p-1, \;p$ which correspond to the endpoints of some spring assumed to be non linear:
\begin{equation}
  \label{eq:Phi=}
  \Phi_{p-1}(u,\epsilon)=c(u_p-u_{p-1})^2 +\frac{d}{\epsilon}(u_p-u_{p-1})^3, \;
  \Phi_p=- \Phi_{p-1}, ~~ p=2, \dots, n
\end{equation}
If the non linear spring would have been the first or the last one,
the expression of the function $\Phi$ would depend on the boundary
condition; each case would be solved using the same method with slight
changes in some formulas.
In order to get an approximate solution, we are going to write it in the generalized eigenvector basis:
\begin{gather}
\label{eq:vect-val-propre}
  K \phi_k= \omega_k M \phi_k, \text{ with } \phi_k^T M \phi_l= \delta_{kl}, ~~~ k,l=1 \dots, n.
\end{gather}
So we perform the change of function 
\begin{equation}
  \label{eq:u=yphi}
u=\sum_{k=1}^n y_k \phi_k   
\end{equation}

we obtain
\begin{equation}
  \ddot y_k +\omega_k^2 y_k+\phi_k^T \Phi(\sum_{i=1}^n y_i \phi_i,\epsilon)=0, ~~~ k=1 \dots, n.
\end{equation}
As $\Phi$ has only 2 components which are not zero, it can be written
\begin{equation}
\label{eq:yk..Phip}
  \ddot y_k +\omega_k^2 y_k+\left ( \phi_{k,p-1}-\phi_{k,p}\right )\Phi_{p-1}(\sum_{i=1}^n y_i \phi_i,\epsilon)=0,  ~~~ k=1 \dots, n
\end{equation}
or more precisely
\begin{multline}
\label{eq:yk..Phipcd}
  \ddot y_k +\omega_k^2 y_k+\left ( \phi_{k,p-1}-\phi_{k,p}\right ) \Bigg[c \left (\sum_{i=1}^n y_i (\phi_{i,p}-\phi_{i,p-1}) \right)^2 +\\
 \frac{d}{\epsilon} \left (\sum_{i=1}^n y_i( \phi_{i,p}--\phi_{i,p-1}) \right)^3 \Bigg]=0,  ~~~ k=1 \dots, n.
\end{multline}
As for the 1 d.o.f. case, we use a double scale expansion to compute
an approximate  small solution; more precisely, we look for a solution
close to the normal mode of the associated linear system; we denote
this mode by subscript $1$; obviously by permuting the coordinates,
this subscript could be anyone (different of  $p$, this case would
give similar results with slightly different formulas); we set
\begin{equation}
  T_0=\omega_1 t, \quad T_1=\epsilon t
\end{equation}
and  we use  the {\it ansatz}:
\begin{equation}
  y_k=\epsilon y_k^1(T_0,T_1) +\epsilon^2 r_k(T_0,T_1,\epsilon)
\end{equation}
so that
\begin{equation}
 \frac{d^2y_k}{dt^2}=\epsilon \omega_1^2D^2_0y_k^1  +\epsilon^2 [2  \omega_1 D_0D_1 y_k^1 + \omega_1^2 D_0^2r_k ] +\epsilon^3 [D_1^2 y_k^1 +  \mathcal{D}_2 r_k] 
\end{equation}
with
\begin{equation}
  \mathcal{D}_2 r_k = \frac{1}{\epsilon}\left( \frac{d^2 r_k}{d t^2} -\omega_1^2 D_0^2 r_k \right)= 2 \omega_1 D_0D_1 r_k +\epsilon D_1^2 r_k.
\end{equation}
We plug previous expansions into \eqref{eq:yk..Phipcd}.
By identifying the coefficients of the powers of $\epsilon$ in the expansion of \eqref{eq:yk..Phip}, we get:
\begin{align}
\Bigg \{
  \begin{array}[h]{rl}
  &\omega_1^2D_0^2 y_k^1 + \omega_k^2 y_k^1=0, \  k=1 \dots, n\\
 & \omega_1^2 D_0^2 r_k + \omega_k^2 r_k=S_{2,k}, \  k=1 \dots, n \quad \text{ with } \label{eq:D02rk=S2k}\\
  \end{array} 
\end{align}
 to simplify, the manipulations, we  set 
$$\delta_p \phi_l=(\phi_{l,p}-\phi_{l,p-1}),$$ so:
\begin{multline}
  S_{2,k}= \frac{-\delta_p \phi_k}{\epsilon^2} \Phi_{p-1} \left
    (\sum_i(\epsilon y_i^1+\epsilon^2r_i )\phi_i,\epsilon \right)- 2 \omega_1 D_0D_1 y_k^1 -
\epsilon {\cal R}_k
\end{multline}
with 
\begin{equation}
  {\cal R}_k=\left(D_1^2 y_k^1 +{\cal D}_2r_k \right)
\end{equation}
and
\begin{multline}
  S_{2,k}= \frac{-\delta_p \phi_k}{\epsilon^2}  \left [c\left (\sum_i(\epsilon y_i^1+\epsilon^2r_i )\delta_p\phi_i \right)^2 +\frac{d}{\epsilon} \left(\sum_i(\epsilon y_i^1+\epsilon^2r_i )\delta_p\phi_i \right)^3   \right] \\
- 2 \omega_1 D_0D_1 y_k^1 -
\epsilon {\cal R}_k.
\end{multline}
The formula may be expanded
\begin{multline}
    S_{2,k}=-\delta_p \phi_k \left[c\sum_{i,j} y_i^1y_j^1 \delta_p\phi_i \delta_p\phi_j +d \sum_{i,j,l} y_i^1y_j^1 y_l^1\delta_p\phi_i \delta_p\phi_j \delta_p\phi_l \right] \\
-2 \omega_1 D_0D_1 y_k^1-\epsilon R_k\left(y^1,r,\epsilon \right)
\end{multline}

where
\begin{multline}
 R_k(y^1,r,\epsilon)= {\cal R}_k \\
+\delta_p \phi_k \Bigg[
\epsilon c \sum_{i,j} \left (2 y_i^1 r_j  +
\epsilon r_ir_j  \right) \delta_p \phi_i \delta_p\phi_j + \\
\epsilon d \sum_{ijl} \left (3 y_i^1 y_j^1r_l  +3\epsilon  y_i^1 r_jr_l   +3 \epsilon^2  r_i r_jr_l\right )\delta_p \phi_i \delta_p \phi_j  \delta_p \phi_l  
 \Bigg].
\end{multline}

We set
 $\theta(T_0,T_1)= T_0+\beta(T_1)$ and we note that $D_0 \theta=1, \; D_1 \theta=D_1 \beta$; we solve the first set of equations \eqref{eq:D02rk=S2k}, imposing $O(\epsilon)$ initial Cauchy data for $k \neq 1$; we get:
 \begin{equation}
\label{eq:y1=}
   y_1^1=a(T_1)cos(\theta), \; \text{ and }  y_k^1=O(\epsilon), \; k=2 \dots n
 \end{equation}
we put terms involving $y_k^1, \; k \ge 2$ into $R_k$; so we obtain
\begin{multline}
S_{2,1}= -\delta_p \phi_1\left[   c\left ( y_1^1\delta_p\phi_1 \right )^2 +d\left ( y_1^1 \delta_p\phi_1 \right ) ^3  \right] \\
-2 \omega_1 D_0D_1y_1^1 - \epsilon R_1(y^1,r,\epsilon) \text{ and }
% \label{eq:S2=}
\end{multline}
\begin{multline}
S_{2,k}= -\delta_p \phi_k\left[   c\left ( y_1^1\delta_p\phi_1 \right )^2 +d\left ( y_1^1 \delta_p\phi_1 \right ) ^3  \right] \\
 - \epsilon R_k(y^1,r,\epsilon) \text{ for } k \neq 1.
 \label{eq:S2syst=}
\end{multline}
Using \eqref{eq:y1=}, we get:
\begin{multline}
  S_{2,1}= -\delta_p \phi_1\Big[   \frac{ca_1^2}{2}(1+cos(2\theta))\left ( \delta_p\phi_1 \right )^2+\\\frac{da_1^3}{4}\left ( (cos(3\theta)+3cos(\theta)) ( \delta_p\phi_1)^3 \right )   \Big] + \\
2\omega_1(D_1 a_1 sin(\theta)+ a_1D_1 \beta_1 cos(\theta)) - \epsilon  R_1(y^1,r,\epsilon) \text{ and }
\end{multline}
\begin{multline}
  S_{2,k}= -\delta_p \phi_k\Big[   \frac{ca_1^2}{2}(1+cos(2\theta))\left ( \delta_p\phi_1 \right )^2+\\\frac{da_1^3}{4}\left ( (cos(3\theta)+3cos(\theta)) ( \delta_p\phi_1)^3 \right )   \Big] + \\
 - \epsilon  R_k(y^1,r,\epsilon) \text{ for } k \neq 1.
\end{multline}

We gather the terms at angular frequency $1$ in $S_{2,1}$
\begin{multline}
  S_{2,1}= -\delta_p \phi_1\Big[
\frac{da_1^3}{4} 3cos(\theta)(\delta_p\phi_1)^3   \Big]  \\
+2\omega_1(D_1 a_1 sin(\theta)+ a_1D_1 \beta_1 cos(\theta)) + S_{2,1}^{\sharp} - \epsilon  R(y^1,r,\epsilon)
\end{multline}
with
 \begin{multline}
   S_{2,1}^{\sharp}= -\delta_p \phi_1 \left[  \frac{ca_1^2}{2}(1+cos(2\theta)) ( \delta_p\phi_1)^2  +\frac{da_1^3}{4} cos(3\theta)(\delta_p\phi_1)^3  \right].
\end{multline}
If we enforce
\begin{multline}
  \label{eq:D1a1D1beta1}
  D_1a_1=0, ~\text{ and }~ 2\omega_1a_1D_1 \beta_1
  =(\delta_p\phi_1)^4\frac{3da^3}{4} \text{ so that  } \\ a_1=a_{1,0},
  \quad \beta_1=\beta_{1,0}T_1 \text{ with } \beta_{1,0}=\frac{3da^2}{8 \omega}(\delta_p\phi_1)^4 T_1
\end{multline}
the right hand side
\begin{equation}
  S_{2,1}= S_{2,1}^{\sharp}-\epsilon R_1(y^1,r,\epsilon)
\end{equation}
contains no term at angular frequency $1$; 
for the other components, without any manipulation,  there is no
trouble with the frequencies if we assume that all the
eigenfrequencies $\omega_k$ for $k=2 \dots n$ are not multiple of $\omega_1$
($\omega_k \neq q \omega_1$ for $q=1$ or $q=2$, $q=3$).

In order to prove that $r$ is bounded, after the elimination of terms at frequency $1$, we write back the equations with the variable $t$, for the second set of equations of \eqref{eq:D02rk=S2k}.
\begin{equation}
  \omega_1^2 \ddot r_k  +\omega_k^2 r_k=\tilde S_{2,k} \quad \text{ for } k=1, \dots n
\end{equation}
with
\begin{equation}
  \tilde S_{2,1}=S_{2,1}^{\sharp}-\epsilon \tilde R_1(y^1,r,\epsilon)
\end{equation}
where
 \begin{multline}
   S_{2,1}^{\sharp}= -\delta_p \phi_1 \bigg [  \frac{ca_1^2}{2}(1+cos(2(\omega_1t+\beta_{1,0}\epsilon t)) ( \delta_p\phi_1)^2  \\
 +\frac{da_1^3}{4} cos(3(\omega_1t+\beta_{1,0}\epsilon t))(\delta_p\phi_1)^3  \bigg]
\end{multline}

and
\begin{multline}
  \tilde S_{2,k}= -\delta_p \phi_k\Bigg[   \frac{ca_1^2}{2}(1+cos(2 (\omega_1t+\beta_{1,0}\epsilon t) ))\left ( \delta_p\phi_1 \right )^2+\\\frac{da_1^3}{4}\left ( (cos(3(\omega_1t+\beta_{1,0}\epsilon t))+3cos((\omega_1t+\beta_{1,0}\epsilon t))) ( \delta_p\phi_1)^3 \right )   \Bigg]  \\
 - \epsilon  \tilde R_k(y^1,r,\epsilon) \text{ for } k \neq 1
\end{multline}
and where
\begin{equation}
  \tilde R_k(y^1,r,\epsilon)= R_k(y^1,r,\epsilon)-{\cal D}_2 r_k
\end{equation}

\begin{proposition}
  Under the assumption that $\omega_k$ and $\omega_1$ are $\mathbb{Z}$ independent for $k \ne 1$, there exists $\gamma>0$ such that for all $t \le t_{\epsilon} =\frac{\gamma}{\epsilon}$, the solution of \eqref{eq:yk..Phip} with initial data
  \begin{equation}
y_1(0)    =\epsilon a_{1,0}, \quad \dot y_1(0)=0, \quad y_k(0)=O(\epsilon^2), \quad \dot y_k(0)=0
  \end{equation}
satisfy the following expansion
\begin{align}
  y_1(t)&= \epsilon a_0cos(\nu_{\epsilon}t)+\epsilon^2r_1(\epsilon,t) \text{ with } \nu_{\epsilon}=\omega_1+3\epsilon\frac{da_0^2}{8\omega_1}(\phi_{1,p}-\phi_{1,p-1})^4\\
y_k(t)&=\epsilon^2r_k(\epsilon,t)
\end{align}
with $r_k$ uniformly bounded in ${\cal C}^2(0,t_{\epsilon})$ for $k=1,\dots n$
and $\omega_1, \phi_1$ are the eigenvalue and eigenvectors defined in \eqref{eq:vect-val-propre}.
\end{proposition}
\begin{corollary}
  The solution of \eqref{eq:system}, \eqref{eq:Phi=} with 
$$\phi_1^T u(0)=\epsilon a_{1,0}, ~~ \phi_1^T \dot u(0)=0, ~~ \phi_k^T u(0)=O(\epsilon^2), ~~ \phi_k^T  \dot u(0)=0$$
with $\omega_k, \phi_k$ are the eigenvalue and eigenvectors defined in \eqref{eq:vect-val-propre}
\begin{equation}
\text{ is }  ~~~ u(t)=\sum_{k=1}^n y_k(t) \phi_k
\end{equation}
with the expansion of $y_k$ of previous proposition.
\end{corollary}
\begin{proof}
  For the proposition, we use  lemma   \ref{eq:lemmew-syst}. Set
  $S_1=S_{2,1}^{\sharp}, ~~ S_k=S_{2,k} $ for $k=1, \dots n$; as we
  have enforced \eqref{eq:D1a1D1beta1}, the functions $S_k$ are
  periodic, bounded, and are orthogonal to $e^{\pm it}$, we have
  assumed  that $\omega_k$ and $\omega_1$ are $\mathbb{Z}$ independent
  for $k \ne 1$; so  $S_k, ~ k=1,\dots, n$ satisfies the lemma hypothesis. Similarly,
  set $g=\tilde R$, its components are polynomials in $r$ with coefficients which are bounded functions, so it is lipschitzian on the bounded subsets of $\mathbb R$, it satisfies the hypothesis of  lemma   \ref{eq:lemmew-syst} and so the proposition is proved.
The corollary is an easy consequence of the proposition and the change of function \eqref{eq:u=yphi}
\end{proof}
\begin{rem}
  We have obtained a periodic asymptotic expansion of a solution of
  system \eqref{eq:system}, \eqref{eq:Phi=}; they are called non linear normal
  modes in the mechanical community (\cite{nnm-kpgv,
    jiang-pierre-shaw04}. In the next section, we shall derive that
  the frequencies of the normal mode are resonant frequencies for an
  associated forced system, the so called primary resonance; secondary
  resonance could be derived along similar lines.
\end{rem}

\subsection{Forced, damped vibrations, double scale expansion}

\subsubsection{Derivation of the expansion}

We consider a similar system of forced  vibrating masses attached to
springs with a light damping:
\begin{equation}
  \label{eq:system-force}
  M\ddot u + \epsilon C \dot u +K u +\Phi(u,\epsilon)=\epsilon^2 F cos(\tilde \omega_{\epsilon} t)
\end{equation}
with the same assumptions as in subsection
\ref{subsec:syst-free-vib}. We assume that the frequency of the
driving force is close to some frequency of the linearised system
(primary resonance); we
denote this frequency with the subscript $1$: $\tilde
\omega_{\epsilon}=\omega_1+\epsilon \sigma.$

We assume that the non linearity is local, all components are zero except for two components $p-1, \;p$ which correspond to the endpoints of some spring assumed to be non linear.
As for free vibrations, we perform  the change of function 
\begin{equation}
  \label{eq:u=yphi}
u=\sum_{k=1}^n y_k \phi_k   
\end{equation}
with $\phi_k$, the generalised eigenvectors of
\eqref{eq:vect-val-propre}.
As the damping matrix $C$ is usually not well defined, to simplify, we
assume that it is diagonal in the eigenvector basis $\phi_k, \; k=1,
\dots n$.
We obtain
\begin{equation}
  \ddot y_k + \epsilon \lambda_k \dot y_k +\omega_k^2 y_k+\phi_k^T
  \Phi(\sum_{i=1}^n y_i \phi_i, \epsilon)=\epsilon^2 f_k cos(\tilde \omega_{\epsilon} t) , ~~~ k=1 \dots, n
\end{equation}
with $f_k=\phi_k^T F$. 
As for the free vibration case,  $\Phi$ has only 2 components which
are not zero, so the system can be written
\begin{multline}
\label{eq:yk..Phip-force}
  \ddot y_k + \epsilon \lambda_k \dot y_k+\omega_k^2 y_k+\left (
    \phi_{k,p-1}-\phi_{k,p}, \epsilon \right )\Phi_{p-1}(\sum_{i=1}^n y_i
  \phi_i)=\epsilon^2 f_k cos(\tilde \omega_{\epsilon} t) ,  \\ ~~~ k=1 \dots, n
\end{multline}
or more precisely
\begin{multline}
\label{eq:yk..Phipcdforc}
  \ddot y_k +\epsilon \lambda_k \dot y_k+\omega_k^2 y_k+\left ( \phi_{k,p-1}-\phi_{k,p}\right ) \Bigg[c \left (\sum_{i=1}^n y_i (\phi_{i,p}-\phi_{i,p-1}) \right)^2 +\\
 \frac{d}{\epsilon} \left (\sum_{i=1}^n y_i( \phi_{i,p}-\phi_{i,p-1})
 \right)^3 \Bigg]=\epsilon^2 f_k cos(\tilde \omega_{\epsilon} t),  \\~~~ k=1 \dots, n.
\end{multline}
As for the 1 d.o.f. case, we use a double scale expansion to compute
an approximate  small solution; we use a fast scale which is $\epsilon$ dependent; we set
\begin{equation}
  T_0=\tilde \omega_{\epsilon} t, \quad T_1=\epsilon t
\end{equation}
and we use the {\it ``ansatz''}
\begin{equation}
  y_k=\epsilon y_k^1(T_0,T_1) +\epsilon^2 r_k(T_0,T_1,\epsilon)
\end{equation}
so that

\begin{equation}
  \label{eq:dykdt=}
 \frac{dy_k}{dt}=  \epsilon \big [  \omega_1 D_0y_k^1+ \epsilon \sigma D_0
 y_k^1 +\epsilon D_1 y_k^1 \big ] +  \epsilon^2 \omega_1 D_0 r_k +\epsilon^2 (\frac{dr_k}{dt} -\omega_1 D_0 r_k)
\end{equation}

\begin{multline}
 \frac{d^2 y_k}{dt^2}=\epsilon \bigg \{  \omega_1^2 D_0^2 y_k^1
 +2\epsilon \omega_1
 \left  [ \sigma D_0^2 y_k^1 +   D_0D_1 y_k^1  \right]+ \\
\epsilon^2 \left [ \sigma^2 D_0^2 y_k^1 +2\sigma D_0D_1y_k^1+D_1^2 y_k^1
\right ] \bigg  \}\\
+\epsilon^2 \omega_1^2 D_0^2r_k  +\epsilon^3 \mathcal{D}_2 r_k
\end{multline}
with
\begin{multline}
  \mathcal{D}_2 r_k = \frac{1}{\epsilon}\left( \frac{d^2 r_k}{d t^2}
    -\omega_1^2 D_0^2 r_k \right)= 2\omega_1( \sigma D_0^2r_k+    D_0D_1 r_k) \\
+  \epsilon  \left[ \sigma^2 D_0^2 r_k +2\sigma D_0D_1r_k+D_1^2 r_k \right].
\end{multline}
We plug previous expansions into \eqref{eq:yk..Phipcdforc}.
By identifying the coefficients of the powers of $\epsilon$ in the expansion of \eqref{eq:yk..Phipcdforc}, we get:
\begin{align}
\Bigg \{
  \begin{array}[h]{rl}
  &\omega_1^2D_0^2 y_k^1 + \omega_k^2 y_k^1=0, \  k=1 \dots, n\\
 & \omega_1^2 D_0^2 r_k + \omega_k^2 r_k=S_{2,k}, \  k=1 \dots, n \quad \text{ with } \label{eq:D02rk=S2kforce}\\
  \end{array} 
\end{align}

\begin{multline}
 S_{2,k}= 
-\Bigg \{ \frac{\delta_p \phi_k}{\epsilon^2} \Phi_{p-1}
   \left(\sum_i (\epsilon y_i^1+\epsilon^2r_i )\phi_i,\epsilon \right)+ 2 \omega_1 
[ D_0D_1 y_k^1 +\sigma D_0^2 y_k^1 ] +\lambda_k \omega_1 D_0 y_k^1
 \Bigg \}  \\
+f_k cos(T_0) -\epsilon R_k(y^1,r,\epsilon) 
\end{multline}
where we gather higher order terms in $R_k$ and to simplify, the
manipulations, we have set 
$$\delta_p \phi_l=(\phi_{l,p}-\phi_{l,p-1}),$$ so:
\begin{multline}
  S_{2,k}= -  \frac{\delta_p \phi_k}{\epsilon^2}  \left [c\left (\sum_i(\epsilon y_i^1+\epsilon^2r_i )\delta_p\phi_i \right)^2 +\frac{d}{\epsilon} \left(\sum_i(\epsilon y_i^1+\epsilon^2r_i )\delta_p\phi_i \right)^3   \right] \\
- 2 \omega_1 [D_0D_1 y_k^1 +\sigma D_0^2 y_k^1 ] -\lambda_k \omega_1 D_0 y_k^1   \\
+f_k cos(T_0) - \epsilon R_k(y^1,r,\epsilon).
\end{multline}
The formula may be expanded
\begin{multline}
    S_{2,k}=-\delta_p \phi_k \left[c\sum_{i,j} y_i^1y_j^1 \delta_p\phi_i \delta_p\phi_j +d \sum_{i,j,l} y_i^1y_j^1 y_l^1\delta_p\phi_i \delta_p\phi_j \delta_p\phi_l \right] \\
- 2 \omega_1 [D_0D_1 y_k^1 +\sigma D_0^2 y_k^1 ] -\lambda_k \omega_1 D_0 y_k^1   \\
+f_k cos(T_0) - \epsilon R_k(y^1,r,\epsilon)
\end{multline}

% where
% \begin{multline}
%  R_k(y^1,r,\epsilon)= {\cal D}_2 r_k + D_1^2y_k^1 \\
% +\delta_p \phi_k \Bigg[
% \epsilon c \sum_{i,j} \left (2 y_i^1 r_j  +
% \epsilon r_ir_j  \right) \delta_p \phi_i \delta_p\phi_j + \\
% \epsilon d \sum_{ijl} \left (3 y_i^1 y_j^1r_l  +3\epsilon  y_i^1 r_jr_l   +3 \epsilon^2  r_i r_jr_l\right )\delta_p \phi_i \delta_p \phi_j  \delta_p \phi_l  
%  \Bigg] +........
% \end{multline}

We set
 $\theta(T_0,T_1)= T_0+\beta(T_1)$ and we note that $D_0 \theta=1, \;
 D_1 \theta=D_1 \beta$; we solve the first set of equations
 \eqref{eq:D02rk=S2kforce}, imposing  initial Cauchy data for $k
 \neq 1$ of order $O(\epsilon)$
%(up to the order $\epsilon^2$); 
we get:
 \begin{equation}
\label{eq:y1=syst}
   y_1^1=a_1(T_1)cos(\theta), \; \text{ and }  y_k^1=O(\epsilon), \; k=2 \dots n
 \end{equation}
we put terms involving $y_k^1$ into $R_k$ for $k \ge 2$ and so we obtain
\begin{multline}
S_{2,1}= -\delta_p \phi_1\left[   c\left ( y_1^1\delta_p\phi_1 \right )^2 +d\left ( y_1^1 \delta_p\phi_1 \right ) ^3  \right] \\
-2 \omega_1 [ D_0D_1y_1^1 +\sigma D_0^2 y_1^1] -\lambda_1 \omega_1D_0 y_1^1  +f_1\cos(T_0) - \epsilon R_1(y^1,r,\epsilon) \text{ and }
% \label{eq:S2=}
\end{multline}
\begin{multline}
S_{2,k}= -\delta_p \phi_k\left[   c\left ( y_1^1\delta_p\phi_1 \right )^2 +d\left ( y_1^1 \delta_p\phi_1 \right ) ^3  \right] +\\
f_k\cos(T_0) - \epsilon R_k(y^1,r,\epsilon) \text{ for } k \neq 1.
 \label{eq:S2systf=}
\end{multline}
Using \eqref{eq:y1=syst}, we get:
\begin{multline}
  S_{2,1}= -\delta_p \phi_1\Big[   \frac{ca_1^2}{2}(1+\cos(2\theta))\left ( \delta_p\phi_1 \right )^2+\\\frac{da_1^3}{4}\left ( (\cos(3\theta)+3\cos(\theta)) ( \delta_p\phi_1)^3 \right )   \Big] + \\
2\omega_1[ D_1 a_1 sin(\theta)+ a_1D_1 \beta_1 \cos(\theta) +\sigma a_1
\cos(\theta)  ] +\lambda_1\omega_1 a_1 sin(\theta)\\
+f_1 [\sin(\theta)\sin(\beta)+\cos(\theta)\cos(\beta) ]- \epsilon  R_1(y^1,r,\epsilon) \text{ and }
\end{multline}
\begin{multline}
  S_{2,k}= -\delta_p \phi_k\Big[   \frac{ca_1^2}{2}(1+\cos(2\theta))\left ( \delta_p\phi_1 \right )^2+\\\frac{da_1^3}{4}\left ( (\cos(3\theta)+3\cos(\theta)) ( \delta_p\phi_1)^3 \right )   \Big] + \\
+f_k [\sin(\theta)\sin(\beta)+\cos(\theta)\cos(\beta) ] - \epsilon  R_k(y^1,r,\epsilon) \text{ for } k \neq 1.
\end{multline}

We gather the terms at angular frequency $1$ in $S_{2,1}$
\begin{multline}
  S_{2,1}= \delta_p \phi_1\Bigg[
-3\frac{da_1^3}{4} \cos(\theta)(\delta_p\phi_1)^3   +
  2\omega_1(a_1D_1\beta_1+\sigma a_1) + f_1 \cos(\beta)  \Bigg] \cos(\theta)\\
+ \Big [  \omega_1(  2D_1 a_1 + \lambda_1 a_1 ) +f_1 sin(\beta) \Big]  sin(\theta)+ S_{2,1}^{\sharp} - \epsilon  R(y^1,r,\epsilon)
\end{multline}
with
 \begin{multline}
   S_{2,1}^{\sharp}= -\delta_p \phi_1 \left[  \frac{ca_1^2}{2}(1+\cos(2\theta)) ( \delta_p\phi_1)^2  +\frac{da_1^3}{4} \cos(3\theta)(\delta_p\phi_1)^3  \right].
\end{multline}
\paragraph{Orientation}

If we enforce
\begin{multline}
  \label{eq:D1a1D1beta1syst}
\Bigg \{
\begin{array}[h]{rl}
  \omega_1 \big (2D_1a_1+\lambda_1 a_1 \big)&=-f_1 sin(\beta_1), ~~\text{ and }\\
~ 2\omega_1  \big ( a_1D_1 \beta_1 + \sigma a_1  \big )
  &=\frac{3da^3}{4} (\delta_p\phi_1)^4 -f_1cos(\beta_1) 
\end{array}
\end{multline}

 the right hand side
\begin{equation}
  S_{2,1}= S_{2,1}^{\sharp}-\epsilon R_1(y^1,r,\epsilon)
\end{equation}
contains no term at angular frequency $1$; 
for the other components, without any manipulation,  there is not such terms
, if we assume that all the
eigenfrequencies $\omega_k$ for $k=2 \dots n$ are not multiple of $\omega_1$
($\omega_k \neq q \omega_1$ for $q=1$ or $q=2$, $q=3$).
This will enable us to justify this expansion; previously, we study
the stationary solution of this approximate system and the stability of the
solution in a neighbourhood of this stationary solution. 

\subsubsection{Stationary solution and stability}
The situation is very close to the 1 d.o.f. case; except the
replacement of $d$ by  of $\tilde d =d(\delta_p \phi_1)^4$, the system
\eqref{eq:D1a1D1beta1syst} is the same as \eqref{eq:D1aD1beta1}; the
other components are zero. We state a similar proposition

\begin{proposition}
  When  $\sigma \le \frac{3 \tilde d \bar a^2}{4 \omega}
  -\frac{1}{2}\sqrt{\frac{9\tilde d^2 \bar a^4}{16 \omega^2}-\lambda_1^2}$
, the stationary solution of  \eqref{eq:D1a1D1beta1syst}
is stable in the  sense of Lyapunov (if the dynamic solution starts
close to the stationary one, it remains close and converges to it); to
the stationary case corresponds the approximate  solution of
\eqref{eq:yk..Phipcd} $y_1^1=\bar a_1 \cos(T_0+ \bar \beta_1), ~~ y_k^1=O(\epsilon),
~~ k=2, \dots, n$, it is periodic; for an initial data close enough to
the stationary solution, $y_1^1= a(T_1) \cos(T_0+  \beta_1(T_1)), ~~ y_k^1=O(\epsilon),
~~ k=2, \dots, n$ with $a,\beta_1$ solutions of  \eqref{eq:D1a1D1beta1syst}
with $d$ replaced by $\tilde d$
; they converge to the  stationary solution  $\bar a_1, \bar \beta_1$  when $T_1 \longrightarrow +\infty$.
\end{proposition}

\subsubsection{Convergence of the expansion}

In order to prove that $r$ is bounded, after the elimination of terms at frequency $1$, we write back the equations with the variable $t$, for the second set of equations of \eqref{eq:D02rk=S2k}.
\begin{equation}
  \omega_1^2 \ddot r_k  +\omega_k^2 r_k=\tilde S_{2,k} \quad \text{ for } k=1, \dots n
\end{equation}
with
\begin{equation}
  \tilde S_{2,1}=S_{2,1}^{\sharp}-\epsilon \tilde R_1(y^1,r,\epsilon)
\end{equation}
where
 \begin{multline}
   S_{2,1}^{\sharp}= -\delta_p \phi_1 \bigg [  \frac{c(a_1(\epsilon t))^2}{2}(1+\cos(2(\tilde \omega_{\epsilon}t+\beta_{1}(\epsilon t)) ( \delta_p\phi_1)^2  \\
 +\frac{da_1^3}{4} \cos(3(\tilde \omega_{\epsilon}t+\beta_{1}(\epsilon t))(\delta_p\phi_1)^3  \bigg]
\end{multline}

and
\begin{multline}
  S_{2,k}= -\delta_p \phi_k\Bigg[   \frac{c(a_1(\epsilon t)^2}{2}(1+\cos(2 (\tilde \omega_{\epsilon}t+\beta_{1}(\epsilon t) ))\left ( \delta_p\phi_1 \right )^2+\\\frac{da_1^3}{4}\left ( (\cos(3(\tilde \omega_{\epsilon}t+\beta_(\epsilon t))+3\cos((\tilde \omega_{\epsilon}t+\beta_{1}(\epsilon t))) ( \delta_p\phi_1)^3 \right )   \Bigg]  \\
 - \epsilon  R_k(y^1,r,\epsilon) \text{ for } k \neq 1
\end{multline}
where
\begin{equation}
  \tilde R_k(y^1,r,\epsilon)= R_k(y^1,r,\epsilon)-{\cal D}_2 r_k -
  \lambda_k \bigl (\frac{d r_k}{d t}-\omega_k D_0 r_k \bigr )
\end{equation}

\begin{proposition}
  Under the assumption that $\omega_k$ and $\omega_1$ are $\mathbb{Z}$ independent for $k \ne 1$, there exists $\gamma>0$ such that for all $t \le t_{\epsilon} =\frac{\gamma}{\epsilon}$, the solution of \eqref{eq:yk..Phipcdforc} with initial data
  \begin{gather}
y_1(0)    =\epsilon a_{1,0} +O(\epsilon^2) , \quad \dot y_1(0)=-\epsilon \omega
a_{1,0}sin(\beta_{1,0}) +O(\epsilon^2), \\\quad y_k(0)=O(\epsilon^2), \quad \dot y_k(0)=0
  \end{gather}
and with the initial data close to the stationary solution
$$|a_{1,0} -\bar a_1 | \le \epsilon C_1, ~~ |\beta_{1,0}- \bar \beta_1|
\le \epsilon C_1$$

satisfy the following expansion
\begin{align}
  y_1(t)&= \epsilon a_1(\epsilon t)\cos(\tilde \omega_{\epsilon} t+\beta_1(\epsilon t))+\epsilon^2r_1(\epsilon,t) \text{ with } \\
y_k(t)&=\epsilon^2r_k(\epsilon,t)
\end{align}
with $a_1, \beta_1$ solution of \eqref{eq:D1a1D1beta1syst} and
with $r_k$ uniformly bounded in ${\cal C}^{2}(0,t_{\epsilon})$ for $k=1,\dots n$
and $\omega_1, \phi_1$ are the eigenvalue and eigenvectors defined in
\eqref{eq:vect-val-propre} and $a_1\beta_1$ are solution of \eqref{eq:D1a1D1beta1syst}
\end{proposition}
\begin{corollary}
  The solution of \eqref{eq:system-force}, \eqref{eq:Phi=} with 
$$\phi_1^T u(0)=\epsilon a_{1,0}, ~~ \phi_1^T \dot u(0)=-\epsilon \omega_1
a_{1,0} \sin(\beta_{1,0}), ~~ \phi_k^T u(0)=O(\epsilon^2), ~~ \phi_k^T  \dot u(0)=0$$
with $\omega_k, \phi_k$  the eigenvalues and eigenvectors defined in \eqref{eq:vect-val-propre}.
\begin{equation}
\text{ is }  ~~~ u(t)=\sum_{k=1}^n y_k(t) \phi_k
\end{equation}
with the expansion of $y_k$ of previous proposition.
\end{corollary}
\begin{proof}

  For the proposition, we use  lemma   \ref{eq:lemmew-syst}. Set $S_1=S_{2,1}^{\sharp}, ~~ S_k=S_{2,k} $ for $k=1, \dots n$; as we have enforced \eqref{eq:D1a1D1beta1}, the functions $S_k$ are periodic, bounded, and are orthogonal to $e^{\pm it}$, we have assumed  that $\omega_k$ and $\omega_1$ are $\mathbb{Z}$ independent for $k \ne 1$; so  $S$ satisfies the lemma hypothesis. Similarly, set $g=\tilde R$, it is a polynomial in $r$ with coefficients which are bounded functions , so it is lipschitzian on the bounded subsets of $\mathbb R$, it satisfies the hypothesis of  lemma   \ref{eq:lemmew-syst} and so the proposition is proved.
The corollary is an easy consequence of the proposition and the change of function \eqref{eq:u=yphi}
\end{proof}

\subsubsection{Maximum of the stationary solution}
As equation \eqref{eq:D1a1D1beta1syst} is similar to the equation
\eqref{eq:D1aD1beta1}  of the 1 d.o.f. case, we get also that the
stationary solution reaches its maximum amplitude to the frequency of
the free periodic solution.

Consider the stationary solution of \eqref{eq:D1a1D1beta1syst}, it satisfies

\begin{align}
\Bigg \{
  \begin{array}[h]{rl}
\lambda_1 a_1 \omega_1 &=-f_1 \sin(\beta_1) \\
a \left (2\omega_1\sigma -\frac{3 \tilde{d} a^2}{4} \right) &=-f_1\cos(\beta_1)
\end{array}
\end{align}
 manipulating, we get that $a_1$ is solution of the equation:
\begin{equation}
  f(a_1,\sigma)=\lambda_1^2 a_1^2 \omega^2 +a_1^2 \left (2\omega_1 \sigma-\frac{3\tilde{d} a_1^2}{4} \right)^2-f_1^2=0.
\end{equation}
As for the 1 d.o.f. case, we can state:

\begin{proposition}
  The stationary solution of \eqref{eq:D1a1D1beta1syst}
%,\eqref{eq:D1aD1beta2} 
satisfies
\begin{align}
\Bigg \{
  \begin{array}[h]{rl}
  \lambda_1 a_1 \omega_1 &+f_1 \sin(\beta_1)=0 \\
2a_1 \omega_1\sigma -\frac{3 \tilde{d} a^3}{4} &+f_1\cos(\beta_1)=0
\end{array}
\end{align} 
it reaches its maximum amplitude for $\sigma=\frac{3\tilde d a_1^2}{8\omega_1} $ and $\beta_1=\frac{\pi}{2}+ k \pi$; the excitation is at the frequency 
$$\tilde \omega_{\epsilon}= \omega_1+3\epsilon \frac{\tilde{d} a_1^2}{8
  \omega_1}, ~ \text{ with } ~\tilde{d}=d (\Phi_{1,p}-\Phi_{1,p-1})^4 ~
~ \text{ and } ~ F=\lambda_1 \omega_1 a_1$$
where $\Phi_{1}$ is the eigenvector of the underlying linear system
associated to $\omega_1$; 
$ \tilde \omega_{\epsilon}$ is the frequency of  the free periodic solution  \eqref{eq:nualpha};
for this frequency, the approximation (of the  solution up to the order $\epsilon$)  is periodic:
\begin{align}
  y_1(t)&=\epsilon\frac{f_1}{\lambda_1 \omega_1}\sin(\tilde \omega_{\epsilon}
  t)+\epsilon^2 r(\epsilon,t) \\
y_k(t)&= \epsilon^2 r_k(\epsilon,t)
\end{align}

\end{proposition}
As for the 1 d.o.f. case we can remark the following points.
\begin{rem}
   This value of $\sigma=\frac{3\tilde{d} a_1^2}{8\omega_1} $ is
  indeed smaller than  the maximal value that  $\sigma$ may reach 
  in order that the system be stable and that the previous expansion  converges as
indicated in  proposition \ref{prop:conv-dev-forc}.
\end{rem}
\begin{rem}
  We note also that when the stationary solution reaches its maximum
  amplitude we have $ f_1=\lambda_1 \omega_1 a_1$ and so we can recover the
  damping ratio $\lambda_1$ from such a forced vibration experiment;
  this is a close link with the linear case (see for example
  \cite{geradin-rixen} or the English translation
  \cite{geradin-rixen-eng}). This is quite interesting in practice as
  the damping ratio is usually difficult to measure. Obviously, we can
  recover the damping ratio for other frequencies by performing
  other experiments.

We can also consider this result as a stability of the process used in
the linear case with respect to the appearance of a small non-linearity.
\end{rem}

\subsubsection{Numerical solution}
\label{subsub:numericals}
We consider numerical solution of \eqref{eq:system-force} with \eqref{eq:Phi=};
we have chosen $M=I$; $u=0$ at both ends, so  $K$ is the classical
matrix
$$k
\begin{pmatrix}
  2 & -1 & \hdotsfor[]{3}\\
-1 & 2 & -1 & \hdotsfor{2} \\
0&-1 &2&-1 & \dots \\
\hdotsfor[]{5} \\
 \hdotsfor[]{3} &-1 &2
\end{pmatrix}
;$$
$C=\lambda I$ with $\lambda=1/2$; for numerical balance, we have
computed $\frac{u}{\epsilon}$; with the choice $p=1$ we have
$\Phi_1=\epsilon [c u_1^2 + d u_1^3]$ with $c=1, d=1$.
In figure \ref{fig:phase-space9ddl}, we find 3 curves in phase space
for components $1,3,6$ of the system.
In figure \ref{fig:fourier9ddlc1}, we find the Fourier transform of
the components; some components have the same transform; the graphs
are slightly non symmetric.
\begin{figure}[p] 
\includegraphics[height=7cm, width=10cm]{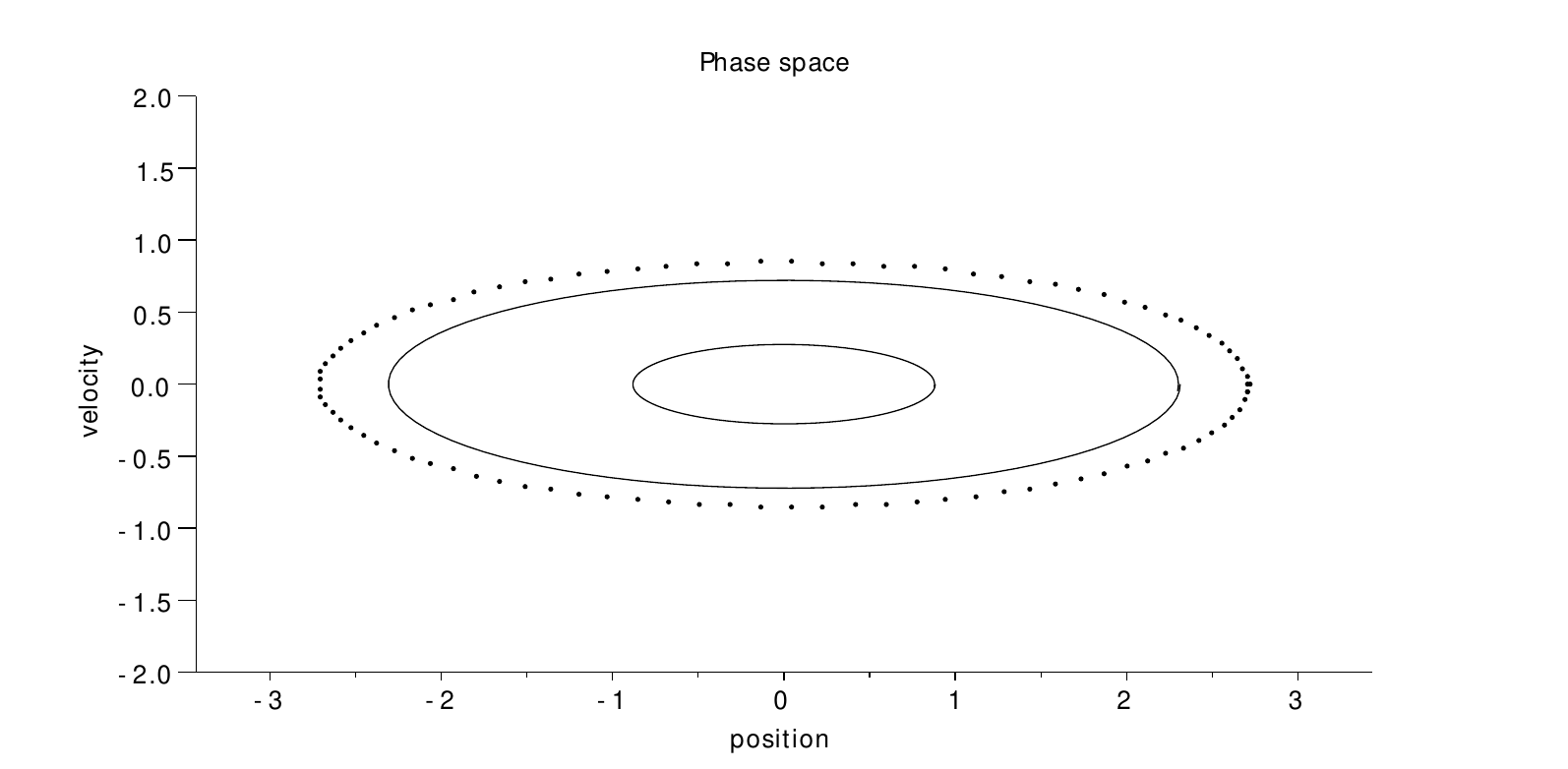}
\caption{Phase portrait of a system with 9 d.o.f. for $ \omega_{\epsilon}= 0.3128868 $}
\label{fig:phase-space9ddl}
\end{figure}
\begin{figure}[p] 
\includegraphics[height=7cm, width=10cm]{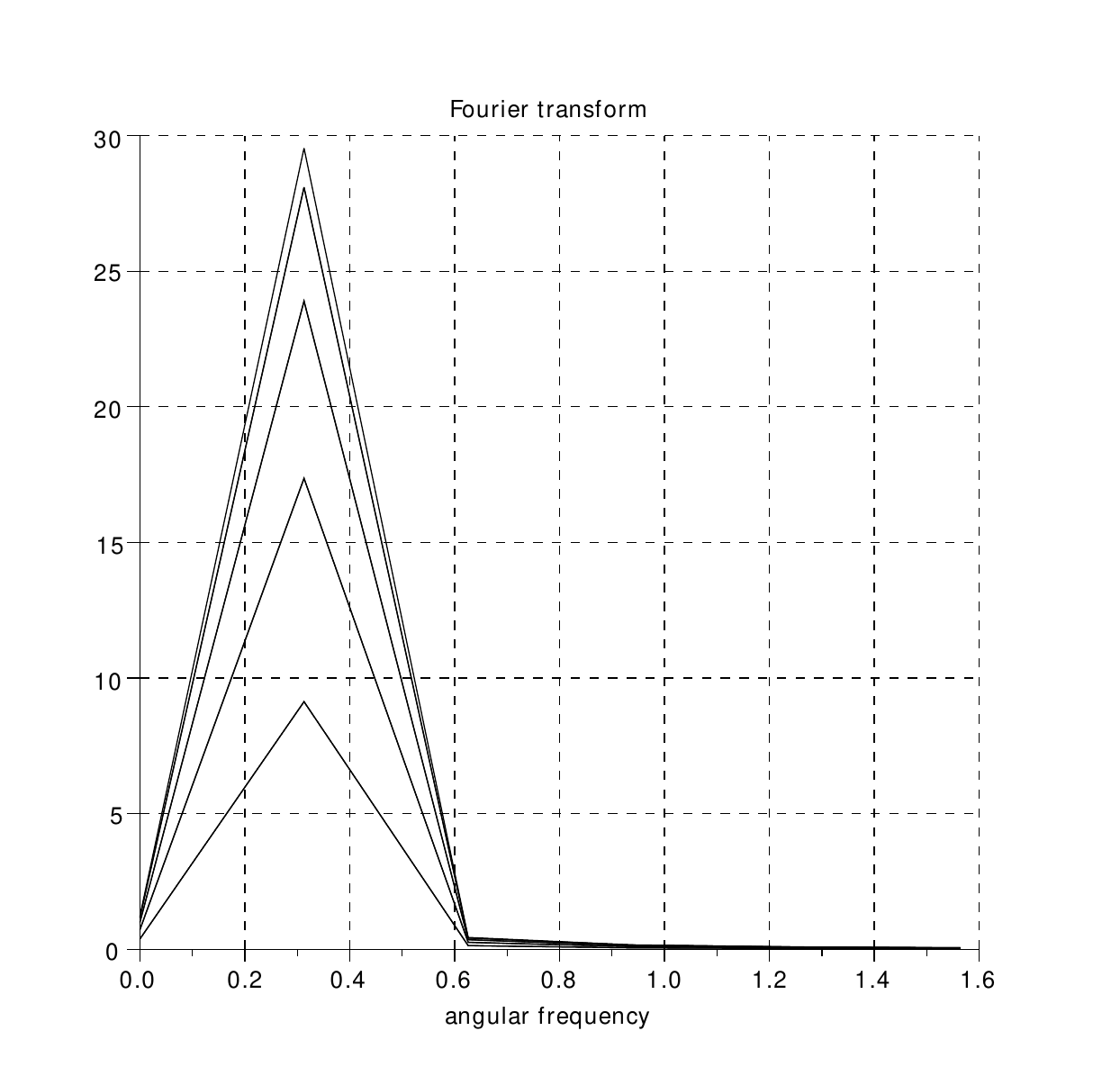}
\caption{Phase portrait of a system with 9 d.o.f. for $ \omega_{\epsilon}= 0.3128868 $}
\label{fig:fourier9ddlc1}
\end{figure}

% %%% Local Variables: 
% %%% mode: latex
% %%% TeX-master:  "doub-appmath-comp"
% %%% End: 

\section{Conclusion}
For differential systems modeling spring-masses vibrations with non
linear springs,  we have derived and rigorously proved a double scale analysis of
periodic solution of free vibrations (so called non linear normal
modes); for damped vibrations with
periodic forcing with frequency close to free vibration frequency ( the
so called primary resonance case), we have obtained an asymptotic
expansion and derived that the amplitude is maximal at the frequency of
the non linear normal mode.
Such non linear vibrating systems linked to a bar
generate acoustic waves; an analysis of the dilatation of a one-dimensional nonlinear crack impacted by a periodic elastic wave,
a smooth model of the crack  may be carried over with a delay
differential equation, \cite{junca-lombard09}. 
\paragraph{Acknowledgment}
We thank S. Junca  for his stimulating interest.

\section{Appendix}
\label{sec:appendix}
\subsection{Inequalities for differential equations}
\begin{lemma}

\label{eq:lemmew }
  Let   $w_{\epsilon}$  be solution of
  \begin{align}
\label{eq:w"+w=}
    w"+w=S(t,\epsilon)+\epsilon g(t,w,\epsilon) \\
w(0)=0, \quad w'(0)=0.
  \end{align}
If the right hand side satisfies the following  conditions
\begin{enumerate}
\item $S$  is a sum of periodic bounded functions:
  \begin{enumerate}
  \item 
for all $t$ and for all $\epsilon$ small enough, $S(t,\epsilon) \le M$
\item $\int_0^{2\pi}e^{i t}S(t,\epsilon)dt=0, \quad \int_0^{2\pi}e^{-i
    t}S(t,\epsilon)dt=0$  uniformly for  $\epsilon$ small enough
  \end{enumerate}
\item for all $R>0$, 
there exists $k_R$ such that for  $|u|\le R$ and $|v|\le R$, the
inequality  $|g(t,u,\epsilon)-g(t,v,\epsilon)|\le k_R |u-v|$ holds and
$|g(t,0,\epsilon)|$is bounded; in other words $g$ is locally
lipschitzian with respect to u.
%%il exists $M_R < +\infty$, such that  $$\sup\{|g(t,w,\epsilon)|, t>0, \;|u|<R, \; \epsilon>0, \text{ assez petit} \} \le M_R $$
\end{enumerate}
then, there exists $\gamma>0$ such that for $\epsilon$ small enough,
$w_{\epsilon}$ is uniformly bounded in $C^{2}(0,T_{\epsilon})$ with $T_{\epsilon}=\frac{\gamma}{\epsilon}$
\end{lemma}
\begin{proof}
The proof  is close  to the proof of  lemma 6.3 of
\cite{junca-br10}; but it is technically simpler since here we assume
$g$ to be locally lipschitzian with respect to $u$ whereas it is only
bounded in \cite{junca-br10}.
  \begin{enumerate}
  \item We first  consider
  \begin{align}
    w_1"+w_1=S(t,\epsilon) \\
w_1(0)=0, \quad w_1'(0)=0
  \end{align}
as  $S$ is  a sum of periodic functions which are  uniformly
orthogonal to $e^{i t}$ and $e^{-i t}$, $w_1$ is bounded in  ${\cal C}^{2}(0,+\infty).$
\item Then we perform a change  of function: $w=w_1+w_2$, the following
  equalities hold
 \begin{align}
    w_2"+w_2= \epsilon g_2(t,w_2,\epsilon)\\
w_2(0)=0, \quad w_2'(0)=0
  \end{align}
with $g_2$ which satisfies the same hypothesis as $g$:

for all  $R>0$,
there exists $k_R$ such that for $|u|\le R$ and $|v|\le R$, the
following inequality holds $|g_2(t,u,\epsilon)-g_2(t,v,\epsilon)|\le k_R |u-v|$.
Using Duhamel principle, the solution of this equation satisfies:
\begin{equation}
  w_2=\epsilon \int_0^t\sin(t-s) g_2(s,w_2(s),\epsilon)ds
\end{equation}
from which
\begin{equation}
  |w_2(t)|\le \epsilon  \int_0^t |g_2(s,w_2(s),\epsilon)- g_2(s,0,\epsilon)|ds +
 \epsilon \int_0^t |g_2(s,0,\epsilon)| ds
\end{equation}
so if $|w| \le R$, hypothesis of lemma imply
\begin{equation}
   |w_2(t)|\le \epsilon \int_0^t  k_R|w_2| ds +\epsilon C t.
\end{equation}
A corollary of lemma of Bellman-Gronwall, see below, will enable to conclude.
It yields
\begin{equation}
  |w_2(t)| \le \frac{ C}{k_R}\left (   \exp(\epsilon k_Rt) -1 \right).
\end{equation}
Now set $T_{\epsilon}=\sup \{t| |w|\le R \}$, then we have
$$R  \le \frac{ C}{k_R}\left (   \exp(\epsilon k_Rt) -1 \right)$$
this shows that there exists $\gamma$ such that $|w_2| \le R$ for $t \le
T_{\epsilon}$, which means that it is  in $L^{\infty}(0,T_{\epsilon})$
for $T_{\epsilon}=\frac{\gamma}{\epsilon}$;  also, we have $w$ in
${\cal C}(0,T_{\epsilon})$ then as $w$ is solution of
\eqref{eq:w"+w=}, it is also bounded in 
  ${\cal C}^{2}(0,T_{\epsilon})$.
  \end{enumerate}
\end{proof}
\begin{lemma}
  (Bellman-Gronwall, \cite{bell-gronw,bellman-perturb}) Let  $u,\epsilon,\beta$ be continuous functions with $\beta \ge 0$, 
  \begin{equation}
    u(t) \le \epsilon(t) + \int_0^t\beta(s)u(s) ds \text{ for } 0\le t\le T
  \end{equation}
then 
\begin{equation}
   u(t) \le \epsilon(t) + \int_0^t \beta(s) \epsilon(s) \left[ \exp(\int_s^t \beta(\tau)d\tau \right ] ds
\end{equation}
\end{lemma}

\begin{lemma} ( a consequence of previous lemma, suited for
  expansions, see \cite{sanders-verhulst})
  Let $u$ be a positive function, $\delta_2 \ge 0$, $\delta_1 >0$ and
$$ u(t) \le \delta_2 t + \delta_1\int_0^tu(s)ds$$ then 
$$u(t) \le \frac{\delta_2}{\delta_1}\left ( exp(\delta_1 t) -1
\right) $$
\end{lemma}

\begin{lemma}
\label{eq:lemmew-syst}
  Let $v_{\epsilon}=[v_1^{\epsilon}, \dots, v_N^{\epsilon}]^T$ be the solution of the following system:
\begin{equation}
\omega_1^2 (v_k^{\epsilon})"+\omega_k^2v_k^{\epsilon}=S_k(t)+\epsilon g_k(t,v_{\epsilon}).
\end{equation}
If   $\omega_1$ and $\omega_k$ are $\mathbb{Z}$ independent for all
$k=2 \dots N$ and 
the right hand side satisfies the following conditions with $M>0, \; C>0$ prescribed constants:
\begin{enumerate}
\item $S_k$ is a sum of bounded periodic functions, $|S_k(t)|\le M$ which satisfy the  non resonance conditions:

  \item $S_1$ is orthogonal to $e^{\pm it}$,
    i.e. $\int_0^{2\pi}S_1(t)e^{\pm it}dt=0$ uniformly for $\epsilon$
    going to zero

\item for all $R>0$ there exists $k_R$ such that for $\|u\| \le R$, $\|v\|\le R$, the following inequality holds for $k=1, \dots, N$ :
$$|g_k(t,u,\epsilon)-g_k(t,v,\epsilon)|\le k_R \|u-v\|$$ and $|g_k(t,0,\epsilon)|$ is bounded
\end{enumerate}
then there exists $\gamma>0$ such that for $\epsilon$ small enough
$v_{\epsilon}$ is bounded in ${\cal C}^{2}(0,T_{\epsilon})$ with $T_{\epsilon}=\frac{\gamma}{\epsilon}$
\end{lemma}
\begin{proof}
% { \bf no longer to do }\\
 \begin{enumerate}
 \item 
We first  consider the linear system
\begin{gather}
\label{eq:vk1}
 \omega_1^2 (v_{k,1})"+\omega_k^2v_{k,1}=S_k  \\
 v_{k,1}(0)=0  \text{ and }  (v_{k,1})'=0
\end{gather}

 For $k=1$, with hypothesis 1.a,  $S_1$ is a sum of bounded periodic
 functions; it is orthogonal to  $e^{\pm it}$, there is no resonance.
For $k\neq 1$, there is no  resonance as
$\frac{\omega_{k}}{\omega_{1}} \notin \mathbb{Z}$ with  hypothesis
1.b.

So  $v_{k,1}$ belongs to $C^{(2)}$ for $k=1,..., n.$
\item Then we perform a change of function
$$v_{k}^{\epsilon}=v_{k,1}+ v_{k,2}^{\epsilon}$$
and  $v_{k,2}^{\epsilon}$  are solutions of the following system :

\begin{gather}
\label{eq:vk2}
  \omega_1^2 (v_{k,2})"+\omega_k^2v_{k,2}=\epsilon
  g_{k,2}(t,v_{k,2},\epsilon), ~ k=1, \dots, N\\
  v_{k,2}^{\epsilon}(0)=0, ~ (v_{k,2}^{\epsilon})'=0, ~ k=1, \dots, N
\end{gather}
%???avec $$v^{\epsilon}= v_{1}+v_{2}^{\epsilon} , v_{2}^{\epsilon}=(...,v_{k,2}^{\epsilon},...) et v_{1}=(....,v_{k,1},....)$$
with $$ g_{k,2}(t,....,v_{k,2}^{\epsilon},....)=g_{k}(t,...,v_{k,1}+v_{k,2}^{\epsilon},....)$$
where  $~ g_{k,2} ~$  satisfies the same  hypothesis as $g_{k}$:\\
for all $R>0$ there exists $k_R$ such that for  $ \parallel u_{k} \parallel \le R$, $ \parallel v_{k} \parallel \le R$, the following inequality holds for $k=1, \dots, N$ :
\begin{equation}
\label{eq:gk2}
 \parallel g_{k,2}(t,u_{k},\epsilon)-g_{k,2}(t,v_{k},\epsilon)\parallel \le k_R  \parallel u_{k}-v_{k} \parallel.
\end{equation}

Using Duhamel principle, the 
 solution or the equation \eqref{eq:vk2} satisfies:
\begin{equation}
  v_{k,2}^{\epsilon}=\epsilon \int_0^t\sin(t-s) g_{k,2}(s,v_{k,2}^{\epsilon}(s),\epsilon)ds
\end{equation}
so 
\begin{multline}
  \parallel  v_{k,2}^{\epsilon}(t)  \parallel  \le \epsilon
  \int_0^t  \parallel g_{k,2}(s,v_{k,2}^{\epsilon}(s),\epsilon)-
  g_{k,2}(s,0,\epsilon) \parallel  ds + \\
 \epsilon \int_0^t \parallel  g_{k,2}(s,0,\epsilon) \parallel  ds
\end{multline}
so with \eqref{eq:gk2}, we obtain
\begin{equation}
 \parallel  v_{k,2}^{\epsilon}(t) \parallel \le \epsilon \int_0^t  k \parallel v_{k,2}^{\epsilon}(t)  \parallel ds +\epsilon C t
\end{equation}
We shall conclude using  Bellman-Gronwall
 lemma; we obtain

\begin{equation}
\parallel  v_{k,2}(t) \parallel \le \frac{C}{k_R}(exp(\epsilon k_R t)-1)
\end{equation}

this shows that there exists $\gamma$ such that $|v_{k,2}^{\epsilon}| \le R$ for $t \le
T_{\epsilon}$, which means that it is  in $L^{\infty}(0,T_{\epsilon})$
for $T_{\epsilon}=\frac{\gamma}{\epsilon}$;  also, we have $v_k$ in
${\cal C}(0,T_{\epsilon})$ then as $v_k$ is solution of
\eqref{eq:w"+w=}, it is also bounded in 
  ${\cal C}^{2}(0,T_{\epsilon})$.

 \end{enumerate}

\end{proof}

\begin{theorem}
\label{th:poinc-lyapu}
( of Poincar\'e-Lyapunov, for  example see \cite{sanders-verhulst})
Consider the equation 
$$\dot x=(A+B(t))x +g(t,x), \; x(t_0)=x_0, \; t\ge t_0$$
where $x,x_0 \in \mathbf{R}^n  $,  $A$ is a constant matrix $n\times
n$  with all its eigenvalues with negative real parts; $B(t)$ is a
matrix which is  continuous  with the property $\lim_{t \rightarrow +\infty} \|B(t)\|=0$.
The vector field  is continuous with respect to  $t$ and $x$ is
continuously differentiable with respect to $x$ in a neighborhood of $x=0$; moreover 
$$g(t,x)= o(\|x\|) \text{ when } \; \|x\|\rightarrow 0 $$ uniformly in $t$.
Then, there exists  constants $C,t_0,\delta,\mu$ such that if $\|x_0\|<\frac{\delta}{C}$ 
$$\|x\| \le C\|x_0\|e^{-\mu(t-t_0)}, t \ge t_0$$ holds
\end{theorem}

\subsection{Numerical computations of Fourier transform }
Assuming a function $f$ to be almost-periodic, the fourier
coefficients are :
\begin{equation}
  \label{eq:four-ap}
  \alpha_n=\lim_{T \rightarrow +\infty} \int_0^T f(t) e^{-\lambda_n t} dt
\end{equation}
(for example, see Fourier coefficients of an almost-periodic function in http://www.encyclopediaofmath.org/).
%$http://www.encyclopediaofmath.org/index.php/Fourier\_coefficients\_of\_an\_almost-periodic\_function$
For numerical purposes, we chose $T$ large enough and consider the
Fourier coefficients of a function of period $T$ equal to $f$ in this interval.
%\bibliography{/home/br/bin/latex/biblio,/home/br/bin/latex/these-hamad,/home/br/bin/latex/biblio-vnl}

\begin{thebibliography}{}

\end{thebibliography}


\begin{thebibliography}{KPGV09}

\bibitem[bel]{bell-gronw}
Bellman and {G}ronwall inequality.
\newblock Encyclopedia of Mathematics. URL: http://www.encyclopediaofmath.org/.

\bibitem[Bel64]{bellman-perturb}
R. Bellman.
\newblock {\em Perturbation techniques in mathematics, physics, and
  engineering}.
\newblock Holt, Rinehart and Winston, Inc., New York, 1964.

\bibitem[BM55]{Bogolyu-Mitropo-ru}
N.  N. Bogolyubov and Yu.  A. Mitropol{\cprime}ski{\u\i}.
\newblock {\em Asimptoti\v ceskie metody v teorii neline\u\i nyh kolebani\u\i}.
\newblock Gosudarstv. Izdat. Tehn.-Teor. Lit., Moscow, 1955.

\bibitem[BM61]{Bogolyu-Mitropo-eng}
N.  N. Bogoliubov and Y.  A. Mitropolsky.
\newblock {\em Asymptotic methods in the theory of non-linear oscillations}.
\newblock Translated from the second revised Russian edition. International
  Monographs on Advanced Mathematics and Physics. Hindustan Publishing Corp.,
  Delhi, Gordon and Breach Science Publishers, New York, 1961.

\bibitem[BM62]{Bogolyu-Mitropo-fr}
N.  N. Bogolioubov and I.  A Mitropolski.
\newblock {\em Les m\'ethodes asymptotiques en th\'eorie des oscillations non
  lin\'eaires}.
\newblock Gauthier-Villars \& Cie, Editeur-Imprimeur-Libraire, Paris, 1962.

\bibitem[BR09]{benbrahim-tamtam09}
N. Ben Brahim and B. Rousselet.
\newblock Vibration d'une barre avec une loi de comportement localement non
  lin\'{e}aire.
\newblock In {\em Proceedings of "Tendances des applications math\'ematiques en
  Tunisie, Algerie, Maroc", Morocco (2009)}, pages 479--485, 2009.

\bibitem[BR13]{nbb-br-mult}
N. Ben Brahim and {B}. {R}ousselet.
\newblock Multiple scale expansion of peri0dic solutions of some nonlinear
  vibrating systems.
\newblock {\em in preparation}, 2013.

\bibitem[Bra]{benbrahimSmai}
N. Ben Brahim.
\newblock Vibration d'une barre avec une loi de comportement localement non
  lin\'eaire.
\newblock Communication au Congr\`es Smai 2009.

\bibitem[Bra10]{benbrahimGdrafpac}
N. Ben Brahim.
\newblock Vibration of a bar with a law of behavior locally nonlinear.
\newblock Affiche au GDR-AFPAC conference, 18-22 janvier 2010.

\bibitem[Gas]{gasmiSmai}
A. Gasmi.
\newblock M\'ethode de la moyenne et de double \'echelle pour syst\`eme de
  cordes en vibration non lin\'eaire.
\newblock Communication au Congr\`es Smai 2009.

\bibitem[GR93]{geradin-rixen}
M.  G\'eradin and D.  Rixen.
\newblock {\em Th\'eorie des vibrations. Application \`a la dynamique des
  structures.}
\newblock Masson, 1993.

\bibitem[GR97]{geradin-rixen-eng}
M.  G\'eradin and D.  Rixen.
\newblock {\em Mechanical vibrations : theory and application to structural
  dynamics}.
\newblock Chichester: Wiley, 1997.

\bibitem[Haz]{hazimSmai}
H.  Hazim.
\newblock Frequency sweep for a beam system with local unilateral contact
  modeling satellite solar arrays.
\newblock Communication au Congr\`es Smai 2009.

\bibitem[Haz10]{hazim-these}
H.  Hazim.
\newblock {\em Vibrations of a beam with a unilateral spring. Periodic
  solutions - Nonlinear normal modes}.
\newblock PhD thesis, U. Nice Sophia-Antipolis, J.A. Dieudonn\'e mathematical
  laboratory, 06108, Nice Cedex France, July 2010.
\newblock http://tel.archives-ouvertes.fr/tel-00520999/fr/.

\bibitem[HFR09]{hazim-ecssmt}
H.  Hazim, N.  Fergusson, and B.  Rousselet.
\newblock Numerical and experimental study for a beam system with local
  unilateral contact modeling satellite solar arrays.
\newblock In {\em Proceedings of the 11th European spacecraft structures,
  materials and mechanical testing conference (ECSSMMT 11)}, 2009.
\newblock http://hal-unice.archives-ouvertes.fr/hal-00418509/fr/.

\bibitem[HR09a]{hr-brgdr08}
H. Hazim and B. Rousselet.
\newblock Finite element for a beam system with nonlinear contact under
  periodic excitation.
\newblock In M. Deschamp A. Leger, editor, {\em Ultrasonic wave propagation in
  non homogeneous media}, springer proceedings in physics, pages 149--160.
  Springer, 2009.
\newblock http://hal-unice.archives-ouvertes.fr/hal-00418504/fr/.

\bibitem[HR09b]{hazim-tamtam09}
H. Hazim and B. Rousselet.
\newblock Frequency sweep for a beam system with local unilateral contact
  modeling satellite solar arrays.
\newblock In {\em Proceedings of "Tendances des applications math\'ematiques en
  Tunisie, Algerie, Maroc", Morocco (2009)}, pages 541--545, 2009.
\newblock http://hal-unice.archives-ouvertes.fr/hal-00418507/fr/.

\bibitem[JL01]{janin-lamarque01}
{{Janin, O.} and {Lamarque, C. H.}},
{Comparison of several numerical methods for mechanical systems with
    impacts.},
\newblock {\em Int. J. Numer. Methods Eng. },
51,
9,
1101-1132,
2001,
.
\bibitem[JBL13]{Bastien-Bernard-lamarque}
{{Bastien, J.} and {Bernardin, F.} and {Lamarque,
    C.H.}},
{Non smooth deterministic or stochastic discrete dynamical systems.
    Applications to models with friction or impact.},
\newblock {Mechanical Engineering and Solid Mechanics Series. London: ISTE;
    Hoboken, NJ: John Wiley \&; Sons. xvi},
2013.

\bibitem[JPS04]{jiang-pierre-shaw04}
D.  Jiang, C.  Pierre, and S.W. Shaw.
\newblock Large-amplitude non-lin{e}ar normal modes of piecewise linear
  systems.
\newblock {\em Journal of sound and vibration}, 2004.

\bibitem[jl09]{junca-lombard09}
   {S. Junca and  B. Lombard},
 \newblock   {Dilatation odf a one dimensional nonlinear crack impacted by a periodic elastic wave },
 \newblock	 {\em SIAM J. Appl. Math},	 {2009},	 {70-3}, {735-761},
  \newblock {http://hal.archives-ouvertes.fr/hal-00339279}.


\bibitem[JR09]{sj-brgdr08}
S.  Junca and B.  Rousselet.
\newblock Asymptotic expansion of vibrations with unilateral contact.
\newblock In M. Deschamp A. Leger, editor, {\em Ultrasonic wave propagation in
  non homogeneous media}, springer proceedings in physics, pages 173--182.
  Springer, 2009.

\bibitem[JR10]{junca-br10}
S.  Junca and B.  Rousselet.
\newblock The method of strained coordinates for vibrations with weak
  unilateral springs.
\newblock {\em The IMA Journal of Applied Mathematics}, 2010.
\newblock http://hal-unice.archives-ouvertes.fr/hal-00395351/fr/.

\bibitem[KPGV09]{nnm-kpgv}
G. Kerschen, M. Peeters, J.C. Golinval, and A.F. Vakakis.
\newblock Nonlinear normal modes, part 1: A useful framework for the structural
  dynamicist.
\newblock {\em Mechanical Systems and Signal Processing}, 23:170--194, 2009.

\bibitem[LL58]{MR0102191}
L. D. Landau and E. M. Lif{\v{s}}ic.
\newblock {\em Mekhanika}.
\newblock Theoretical Physics, Vol. I. Gosudarstv. Izdat. Fiz.-Mat. Lit.,
  Moscow, 1958.

\bibitem[LL60]{MR0120782}
L. D. Landau and E. M. Lifshitz.
\newblock {\em Mechanics}.
\newblock Course of Theoretical Physics, Vol. 1. Translated from the Russian by
  J. B. Bell. Pergamon Press, Oxford, 1960.

\bibitem[LL66]{MR0205515}
L. Landau and E. Lifchitz.
\newblock {\em Physique th\'eorique. {T}ome {I}. {M}\'ecanique}.
\newblock Deuxi\`eme \'edition revue et compl\'et\'ee. \'Editions Mir, Moscow,
  1966.

\bibitem[Mik10]{mikhlin10}
Y.  Mikhlin.
\newblock Nonlinear normal vibration modes and their applications.
\newblock In {\em Proceedings of the 9th Brazilian conference on dynamics
  Control and their Applications}, pages 151--171, 2010.

\bibitem[Mil06]{miller2006}
P.  D. Miller.
\newblock {\em Applied asymptotic analysis}, volume 75 of {\em Graduate Studies
  in Mathematics}.
\newblock American Mathematical Society, Providence, RI, 2006.

\bibitem[Mur91]{murdock91}
J.  A. Murdock.
\newblock {\em Perturbations}.
\newblock A Wiley-Interscience Publication. John Wiley \& Sons Inc., New York,
  1991.
\newblock Theory and methods.

\bibitem[Nay81]{nayfeh81}
A.  H. Nayfeh.
\newblock {\em Introduction to perturbation techniques}.
\newblock J. Wiley, 1981.

\bibitem[Nay86]{nayfeh86}
A.  H. Nayfeh.
\newblock Perturbation methods in nonlinear dynamics.
\newblock In {\em Nonlinear dynamics aspects of particle accelerators ({S}anta
  {M}argherita di {P}ula, 1985)}, volume 247 of {\em Lecture Notes in Phys.},
  pages 238--314. Springer, Berlin, 1986.

\bibitem[oL49]{lyapunov49}
A. M. Lyapunov or Liapounoff.
\newblock {\em The general problem of the stability of motion}.
\newblock Princeton University Press, 1949.
\newblock English translation by Fuller from Edouard Davaux's french
  translation (Probl\`eme g\'en\'eral de la stabilit\'e du mouvement, Ann. Fac.
  Sci. Toulouse (2) 9 (1907)); this french translation is to be found in
  url:http://afst.cedram.org/; originally published in Russian in Kharkov. Mat.
  Obshch, Kharkov in 1892.

\bibitem[Poi99]{poincare92-99}
H. Poincar\'e.
\newblock {\em M\'ethodes nouvelles de la m\'ecanique c\'eleste}.
\newblock Gauthier-Villars, 1892-1899.

\bibitem[Rou11]{rousselet:hal-perio-lip}
B. Rousselet.
\newblock {Periodic solutions of o.d.e. systems with a Lipschitz non linearity}.
\newblock July 2011.

\bibitem[Rub78]{rubenfeld78}
L. A. Rubenfeld.
\newblock On a derivative-expansion technique and some comments on multiple
  scaling in the asymptotic approximation of solutions of certain differential
  equations.
\newblock {\em SIAM Rev.}, 20(1):79--105, 1978.

\bibitem[SV85]{sanders-verhulst}
J.A. Sanders and F. Verhulst.
\newblock {\em Averaging methods in nonlinear dynamical systems}.
\newblock Springer, 1985.

\bibitem[VLP08]{vestroni08}
F. Vestroni, A. Luongo, and A. Paolone.
\newblock A perturbation method for evaluating nonlinear normal modes of a
  piecewise linear two-degrees-of-freedom system.
\newblock {\em Nonlinear Dynam.}, 54(4):379--393, 2008.

\end{thebibliography}

\def\cprime{$'$}

%%\end{article}
\end{document}